\documentclass{amsart}

\usepackage{amssymb}
\usepackage{amsmath}
\usepackage{amsthm}
\usepackage{fullpage}
\usepackage[all]{xy}
\usepackage[hypertex]{hyperref}

\newtheorem{theorem}{Theorem}[section]
\newtheorem{definition}[theorem]{Definition}
\newtheorem{lemma}[theorem]{Lemma}
\newtheorem{corollary}[theorem]{Corollary}
\newtheorem{proposition}[theorem]{Proposition}

\newtheorem*{notheorem}{Theorem}
\newtheorem*{nonotation}{Notation}
\newtheorem*{sublemma}{Sublemma}

\newtheorem*{claim}{Claim}

\theoremstyle{definition}
\newtheorem{remark}[theorem]{Remark}
\newtheorem{example}[theorem]{Example} 
\newtheorem{hyp}[theorem]{Hypothesis}

\newcommand{\ainf}{A_{\mathrm{inf}}({{\mathcal{O}}_{\mathbb{C}_p}}/{\mathcal{O}_K})}
\newcommand{\ainfcohen}{A_{\mathrm{inf}}({{\mathcal{O}}_{\mathbb{C}_p}}/{\mathcal{O}_{{K}_{0}}})}
\newcommand{\aplus}{\widetilde{\mathbb{A}}^{+}}

\newcommand{\bdr}{B_{\mathrm{dR}}^{+}}
\newcommand{\cohen}{K_{0}}
\newcommand{\eplus}{\widetilde{\mathbb{E}}^{+}}
\newcommand{\intcohen}{\mathcal{O}_{K_{0}}}
\newcommand{\ocp}{{\mathcal{O}}_{\mathbb{C}_p}}
\newcommand{\ok}{\mathcal{O}_{K}}
\newcommand{\hok}[1]{\mathcal{O}^{(#1)}_{K}}
\newcommand{\hol}[1]{\mathcal{O}^{(#1)}_{L}}

\newcommand{\cp}{\mathbb{C}_p}

\newcommand{\okbar}{\mathcal{O}_{\overline{K}}}

\newcommand{\hokbar}[1]{\mathcal{O}_{\overline{K}}^{(#1)}}

\newcommand{\kbar}{\overline{K}}

\newcommand{\intring}[1]{\mathcal{O}_{#1}}
\newcommand{\ol}{\mathcal{O}_{L}}
\newcommand{\of}{\mathcal{O}_{F}}
\newcommand{\qp}{\mathbb{Q}_p}
\newcommand{\zp}{\mathbb{Z}_p}

\newcommand{\iplus}{I_{+}}
\newcommand{\hiplus}[1]{I_{+}^{#1}}

\newcommand{\gk}{G_{K}}
\newcommand{\gl}{G_{L}}

\newcommand{\coh}[3]{\mathrm{H}^{#1}(#2,#3)}

\newcommand{\lkahler}{\Omega^1_{\mathcal{O}_L/\mathcal{O}_K}}
\newcommand{\llkahler}{\Omega^1_{\mathcal{O}_{L'}/\mathcal{O}_K}}
\newcommand{\hkahler}[1]{\Omega^{(#1)}_{\mathcal{O}_{\overline{K}}/\mathcal{O}_K}}
\newcommand{\olkahler}{\Omega^{(k)}_{\mathcal{O}_{L}/\mathcal{O}_K}}
\newcommand{\holkahler}[1]{\Omega^{(#1)}_{\mathcal{O}_{L}/\mathcal{O}_K}}
\newcommand{\threekahler}[3]{\Omega^{#1}_{{#2}/{#3}}}

\newcommand{\km}{K[X]_{\mathfrak{m}}}

\newcommand{\mkbar}{\mathfrak{m}_{\overline{K}}}

\newcommand{\lhat}[1]{\widehat{L}_{#1}}
\newcommand{\drcoh}[1]{\mathrm{H}_{dR}^{(#1)}(L/K)}
\newcommand{\olhat}{\mathcal{O}_{\widehat{L}}}
\newcommand{\divi}[1]{(#1)_{\mathrm{div}}}
\newcommand{\jl}[1]{J_{#1}\cap\widehat{L}_{#1}}
\newcommand{\A}{A_{\mathrm{inf}}}

\newcommand{\ak}{A_k}
\newcommand{\AAA}[1]{A_{#1}}
\newcommand{\B}{B}
\newcommand{\Bk}{B_k}
\newcommand{\BB}[1]{B_{#1}}
\newcommand{\I}{I}
\newcommand{\II}[1]{I^{#1}}
\newcommand{\tr}[2]{\mathrm{Tr}_{#1/#2}}
\newcommand{\dif}[2]{\mathfrak{D}_{#1/#2}}

\newenvironment{enuroman}{%
\begin{enumerate}%
}{\end{enumerate}}
\newenvironment{enuRoman}{%
\begin{enumerate}%
}{\end{enumerate}}

\begin{document}
\title[Galois Theory of $\boldsymbol{B_{dR}^{+}}$ in the imperfect residue field case]{Galois Theory of $\boldsymbol{B_{dR}^{+}}$ in the imperfect residue field case}
\author{Shun Ohkubo}
\date{}
\address{Graduate School of Mathematical Sciences, University of Tokyo, 3-8-1 Komaba, Meguro-ku, Tokyo 153-8914, Japan}
\email{shuno@ms.u-tokyo.ac.jp}
\maketitle

\begin{abstract}
We generalize a work of Iovita-Zaharescu on the Galois theory of $\bdr$ to the imperfect residue field case. The proof is based on a structure theorem of Colmez's higher K\"{a}hler differentials.
\end{abstract}

\section*{Introduction}\label{introduction}
Let $K$ be a CDVF of characteristic $(0,p)$ with perfect residue field. Let $\bdr$ be the ring of $p$-adic periods of $K$ defined by Fontaine. It is a complete discrete valuation ring with residue field $\cp$. Let $I$ be the maximal ideal of $\bdr$ and let $B_k=\bdr/I^{k+1}$. Then they are endowed with a topology induced by the $p$-adic topology of $\cp$. For an algebraic extension $L/K$, put $G_L=\mathrm{Gal} (\kbar/L)$ and let $\widehat{L}_k$ $(\mathrm{resp}.\ \widehat{L}_{\infty})$ be the topological closure of $L$ in $B_k$ $(\mathrm{resp}.\ \bdr)$ with respect to this topology.

Iovita-Zaharescu studied in \cite{IZ} the following problem:

\begin{center}
For an algebraic extension $L/K$, does one have $\widehat{L}_k=B_k^{G_L}$ $(\mathrm{resp}.\ \widehat{L}_{\infty}=(\bdr)^{G_L})$ ?
\end{center}

\noindent (In fact, they assume $K=\qp$ and $\mathbb{Q}_p^{\mathrm{ur}}\subset L$, but, by slight modifications, the same proofs work for general local fields $K$ with perfect residue field.) When $k=0$, Ax-Sen-Tate proved that this is always true, namely, $\widehat{L}=\cp^{G_L}$, where $\widehat{\ }$ denotes the $p$-adic completion (actually, \cite{Sen1}, \cite{Tate} proved it in the perfect residue field case and \cite{Ax} proved it in the general residue field case). However, in general, this is not always true: In order that this is true, we need some conditions which involve the canonical derivation $d:\ol=\mathcal{O}_{L}^{(0)}\to\threekahler{1}{\ol}{\ok}=\threekahler{(1)}{\ol}{\ok}$ and its higher analogue $d^{(k)}:\mathcal{O}_{L}^{(k-1)}\to\threekahler{(k)}{\ol}{\ok}$ introduced in \cite{Col}. The main theorem of \cite{IZ} is

\begin{notheorem}[\cite{IZ}, Theorem $0.1$, $0.2$, $4.2$]
Let $\cp (1)$ be the Tate twist of $\cp$ by the cyclotomic character.
\begin{enuRoman}
\item If $\coh{0}{\gl}{\cp (1)}=0$, then $\lhat{k}=B_k^{\gl}$ for all $k$ and $\widehat{L}_{\infty}=(\bdr)^{\gl}$. Moreover, these rings are isomorphic to $\widehat{L}$.

\item If $\coh{0}{\gl}{\cp (1)}\neq 0$, then 
$\lhat{k}=B_k^{\gl}$ if and only if $T_p(\threekahler{(n)}{\ol}{\ok})\neq 0$ $($where $T_p(\threekahler{(n)}{\ol}{\ok})$ denotes the Tate module associated to $\threekahler{(n)}{\ol}{\ok})$ and the canonical derivation $d^{(n)}:\mathcal{O}_{L}^{(n-1)}\to\threekahler{(n)}{\ol}{\ok}$ is surjective for $1\le n\le k$. $\widehat{L}_{\infty}=(\bdr)^{\gl}$ if and only if $\widehat{L}_{k}=B_k^{\gl}$ for all $k$.
\end{enuRoman}
\end{notheorem}

The aim of this paper is to prove a generalization of their result to the case where the residue field of $K$ has a finite $p$-basis. To overcome the technical difficulties caused by imperfectness, we prove a structure theorem of the higher K\"{a}hler differentials, which is one of the new ingredients of this paper. Although Iovita-Zaharescu's proof is very complicated, this structure theorem makes the proofs drastically simple.
\section*{Plan of the paper}
In $\S \ref{pre}$, we define rings of $p$-adic periods of Fontaine and state their basic properties. In $\S \ref{sec:kah}$, we define the higher K\"{a}hler differentials of Colmez and generalize his results in the imperfect residue case. In $\S \ref{sec:good}$, we prove a structure theorem for the higher K\"{a}hlaer differentials. In $\S \ref{sec:main}$, we prove the main theorem (Theorem~\ref{thm:main}) using the results in previous sections. Finally, in $\S \ref{sec:deep}$, we will state a refined version of the main theorem in particular cases and give some examples.

\section*{Acknowledgements}
The author wishes to thank his advisor Professor Atsushi Shiho for his continuing advices and encouragement. The author also thanks to Professor Pierre Colmez for sending an unpublished version of his paper \cite{Col} and allowing the author to include an argument there in this paper.

\section*{Notation}
Throughout this paper, $p$ denotes a fixed prime number. A local field $K$ is a CDVF of mixed characteristic $(0,p)$ with residue field $k_K$ satisfying $[k_K:k_K^p]=p^d<\infty$ and $K$ denotes always a local field if there is no particular mention. Let $\ok$, $\pi_K$, $\mathfrak{m}_K$, $e_K$, $k_K$, $v_K$ be the integer ring, a uniformizer, the maximal ideal, the absolute ramification index, the residue field, the valuation normalized by $v_K(\pi_K)=1$ of $K$ and let $K_0$ be a Cohen subfield of $K$. For an algebraic extension $L/K$, denote its $p$-adic completion by $\widehat{L}$, the relative ramification index by $e_{L/K}$ and put $\gl=\mathrm{Gal}(\kbar/L)$. $\kbar$ denotes an algebraic closure of $K$ and $\cp$ denotes the $p$-adic completion of $\kbar$. $\ocp$ denotes the integer ring of $\cp$ and $v_p$ denotes the $p$-adic valuation of $\cp$ normalized by $v_p(p)=1$. All cohomologies are assumed to be the continuous ones. 

For an abelian group $M$, put $M[p^n]=\ker{(p^n:M\to M)}$, $M_{\mathrm{div}}$ the maximal $p$-divisible subgroup of $M$, $T_pM=\varprojlim{M[p^n]}$, $V_pM=\qp\otimes_{\zp}T_pM$. For a ring $R$ and an $R$-module $M$, put $\mu(M)$ the infimum of the number of generators of $M$ as an $R$-module.

For $n=(n_1,\dotsc,n_d)\in\mathbb{N}^{d}$ and elements $x_1,\dotsc ,x_d$ of a commutative ring $R$, we use multi-index notation, i.e., write $x^n$ for $x_1^{n_1}\dotsm x_d^{n_d}$.

\section{Preliminaries}\label{pre}
\subsection{Basic definitions}
Put $\eplus=\varprojlim_{x\mapsto x^p}{\ocp/p\ocp}$, which is a perfect ring of characteristic $p$. This ring is identified with the set $\{(x^{(n)})\in\ocp^{\mathbb{N}}\mid (x^{(n+1)})^p=x^{(n)}\ \mathrm{for}\  \mathrm{all}\ n\}$ as follows: 
For $x=(x_n)\in\eplus$, choose a lift $\widehat{x_n}\in\ocp$ of $x_n$  for each $n$ and put 
\[
x^{(n)}=\lim_{m\to\infty}{\widehat{x_{n+m}}^{p^m}}.
\]
Then the map $(x_n)\mapsto (x^{(n)})$ gives a bijection. If $x=(x^{(n)})\in\eplus$, let $v_{\mathbb{E}}(x)=v_p(x^{(0)})$. This is a valuation of $\eplus$, for which $\eplus$ is complete. For $x\in\ocp$, $\tilde{x}$ denotes an element of $\eplus$ such that $\tilde{x}^{(0)}=x$.

Put $\aplus=W(\eplus)$ and let $\theta:\aplus\to\ocp$ be the ring homomorphism given by
\[
\theta\left( \sum_{n\in\mathbb{N}}{p^n[x_n]}\right) =\sum_{n\in\mathbb{N}}{p^nx_n^{(0)}}
\]
for $x_n\in\eplus$. Here $W(R)$ denotes the Witt ring of a commutative ring $R$ and $[x]$ denotes the Teichm\"{u}ller lift of $x\in R$. Note that $\ker{\theta}=(p-[\widetilde{p}])$.

Let $\theta:\ok\otimes_{\mathbb{Z}}\aplus\to\ocp$ be the linear extension of the above $\theta$. Put $\ainf=\varprojlim{\ok\otimes_{\mathbb{Z}}\aplus /(p,\ker{\theta})^{k+1}}$ and $\theta:\ainf\to\ocp$ denotes the homomorphism induced by $\theta$. Denote the canonical homomorphism $\ainf [p^{-1}]\to\cp$ also by $\theta$. Put $\iplus=\ker{\theta}\subset\ainf$, $\ak=\ainf/\hiplus{k+1}$.

Put $\bdr=\displaystyle\varprojlim{\ainf [p^{-1}]/(\ker{\theta})^{k+1}}$ and denote the canonical homomorphism $\bdr\to\cp$ also by $\theta$. Put $\I=\ker{\theta}\subset\bdr$, $J_k=\II{k}/\II{k+1}$, $\Bk=\bdr/\II{k+1}$.

\subsection{Some properties of $\boldsymbol{\ainf}$}
For general properties of topological rings, see \cite[Chapitre $0$, $\S 7$]{EGA0} for example.

Recall (\cite[$\S 1$]{Fon1}) that $\ainf$ (resp. $A_k$) is Hausdorff complete with respect to $(p,\iplus)$, $p$, $\iplus$-adic topology (resp. $p$-adic topology) and that $\ainf$ $($resp. $A_k)$ is the universal pro-infinitesimal $\ok$-thickening of $\ocp$ (resp. the universal infinitesimal $\ok$-thickening of $\ocp$ of order $\le k$), which means the initial object in the category consisting of the following objects $(D,\theta)$:
$D$ is an $\ok$-algebra endowed with a surjective $\ok$-homomorphism $\theta:D\to\ocp$ which is $(p,\ker{\theta})$-adically Hausdorff complete $($resp. and $(\ker{\theta})^{k+1}=0$$)$.

In the following, we regard $\ainf$ as an $\ok$-algebra and an $\aplus$-algebra.

\begin{lemma}\label{lem3.1}
The canonical homomorphism
\[
\ok\otimes_{\mathcal{O}_{K_0}}\ainfcohen\to\ainf
\]
is an isomorphism.
\end{lemma}
\begin{proof}[Proof $(${\cite[The proof of Proposition~2.1.5]{Bri1}}$)$]
Let $\theta$ (resp. $\theta_{\ok}$) be the canonical projection $\ainfcohen\to\ocp$ (resp. the linear extension of $\theta$ to $\ok\otimes_{\mathcal{O}_{K_0}}\ainfcohen$). First, we claim that $\ok\otimes_{\mathcal{O}_{K_0}}\ainfcohen$ is ($p,\ker{\theta}_{\ok}$)-adically Hausdorff complete. To prove this, we prove the ($p,\ker{\theta_{\ok}}$)-adic topology and the ($p,1\otimes\ker{\theta}$)-adic topology are the same. Since $\ker{\theta_{\ok}}=(1\otimes\ker{\theta},\pi_K\otimes 1-1\otimes [\widetilde{\pi_K}])$, we have $(p,\ker{\theta_{\ok}})^{e_K}\subset (\pi_K\otimes 1,1\otimes\ker{\theta})$ and $(\pi_K\otimes 1,1\otimes\ker{\theta})^{e_{K/K_0}}\subset (p,1\otimes\ker{\theta})$. This claim implies that $\ok\otimes_{\mathcal{O}_{K_0}}A_{\mathrm{inf}}(\ocp/\mathcal{O}_{K_0})$ is a pro-infinitesimal thickening of $\ocp$ over $\mathcal{O}_{K}$, hence the universality of $\ainf$ induces a map $A_{\mathrm{inf}}(\ocp/\ok)\to\ok\otimes_{\mathcal{O}_{K_0}}A_{\mathrm{inf}}(\ocp/\mathcal{O}_{K_0})$. By the universality of $\ainf$, the composition map
\[
\ainf \to\ok\otimes_{\mathcal{O}_{K_0}}\ainfcohen  \stackrel{\mathrm{can.}}{\to}\ainf
\]
is identity. On the other hand, by the universality of $\ainfcohen$, the composition map
\[
\ok\otimes_{\mathcal{O}_{K_0}}\ainfcohen \stackrel{\mathrm{can.}}{\to}\ainf  \to\ok\otimes_{\mathcal{O}_{K_0}}\ainfcohen\\
\]
is also identity, which implies the assertion.
\end{proof}

We say $t_1,\dotsc,t_d\in\ok$ is a $p$-basis of $K$ if $\overline{t_1},\dotsc,\overline{t_{d}}$ is a $p$-basis of $k_K$. Similarly, $t_1,\dotsc,t_{d'}\in\ol$ is a $p$-basis of $L/K$ if $\overline{t_1},\dotsc,\overline{t_{d'}}$ is a $p$-basis of $k_L/k_K$.

\begin{proposition}\label{prop ainf}
Let $W$ be the Witt ring of the maximal perfect subfield of $k_K$. Fix a $p$-basis $t_1,\dotsc,t_d$ of $K_0$ and put $u_i=t_i-[\widetilde{t_i}]\in\ainfcohen$. Let $\theta:\aplus [\![u_1,\dotsc,u_d]\!]\to\ocp$ be the composite $\aplus [\![u_1,\dotsc,u_d]\!]\stackrel{u_i\mapsto 0}{\to}\aplus\stackrel{\theta}{\to}\ocp$.
\begin{enuroman}
\item\label{prop:ainf1} There exists a unique injective $W$-algebra homomorphism $\intcohen\to\aplus [\![u_1,\dotsc,u_d]\!];t_j\mapsto [\widetilde{t_j}]+u_j$. In the rest of this proposition, fix this $\intcohen$-algebra structure on $\aplus [\![u_1,\dotsc,u_d]\!]$.

\item\label{prop:ainf2} The canonical $\aplus$-algebra homomorphism $\aplus [\![u_1,\dotsc,u_d]\!]\to\ainfcohen$ is an $\intcohen$-algebra homomorphism and induces, for $k\in\mathbb{N}$, $\ok$-algebra isomorphisms
\begin{align*}
\ok\otimes_{\mathcal{O}_{K_0}}\aplus [\![u_1,\dotsc,u_d]\!]/(\ker{\theta_{\ok}})^{k+1} & \to\ainf/(\ker{\theta})^{k+1},\\
\ok\otimes_{\intcohen}\aplus [\![u_1,\dotsc,u_d]\!] & \to \ainf,
\end{align*}
where $\theta_{\ok}:\ \ok\otimes_{\mathcal{O}_{K_0}}\aplus [\![u_1,\dotsc,u_d]\!]\to\ocp$ is the linear extension of $\theta$.

\item\label{prop:ainf3} $\ker{\theta_{\ok}}=(1\otimes u_1,\dotsc,1\otimes u_d,\ \pi_K\otimes 1-1\otimes [\widetilde{\pi_K}])$ and $\{1\otimes u_1,\dotsc,1\otimes u_d,\ \pi_K\otimes 1-1\otimes [\widetilde{\pi_K}]\}$ is an $\ok\otimes_{\mathcal{O}_{K_0}}\aplus [\![u_1,\dotsc,u_d]\!]$-regular sequence.
\end{enuroman}
\end{proposition}
\begin{proof}
(i) Since $\aplus [\![u_1,\dotsc,u_d]\!]$ is $\ker{\theta}=(p-[\widetilde{p}],u_1,\dotsc,u_d)$-adic Hausdorff complete, we can replace $\aplus [\![u_1,\dotsc,u_d]\!]$ by $\aplus [\![u_1,\dotsc,u_d]\!]/(\ker{\theta})^{k+1}$. This is shown in the proof of \cite[Lemme~2.1.3]{Bri1}. 
  
(iii) We regard $\aplus=\aplus [\![u_1,\dotsc,u_d]\!]/(u_1,\dotsc,u_d)$ as an $\intcohen$-algebra. For the first assertion, it has only to prove $\ker{(\theta_{\ok}:\ok\otimes_{\intcohen}\aplus\to\ocp)}=(\pi_K\otimes 1-1\otimes [\widetilde{\pi_K}])$. Moreover, we can check it after taking mod$(\pi_K\otimes 1)$ of both sides since $\ok\otimes_{\intcohen}\aplus$ is $(\pi_K\otimes 1)$-adically Hausdorff complete. Then the assertion is immediate. For the latter assertion, obviously $\{1\otimes u_1,\dotsc,1\otimes u_d\}$ is a regular sequence. Since $\ok\cong\intcohen [X]/Q(X)$ with an Eisenstein polynomial $Q(X)\in\intcohen [X]$ and $Q(X)\in\aplus [X]$ is also an Eisenstein polynomial for the prime ideal $(p)$, $\ok\otimes_{\intcohen}\aplus$ is an integral domain. Therefore $\pi_K\otimes 1-1\otimes [\widetilde{\pi_k}]\neq 0$ is a regular element of $\ok\otimes_{\cohen}\aplus$.

(ii) Note that $\ok\otimes_{\intcohen}\aplus [\![u_1,\dotsc,u_d]\!]$  (resp. $\ainfcohen$) is $\ker{\theta}_{\ok}$-adic (resp. $(\ker{\theta})$-adic) Hausdorff complete. When $K=K_0$, we obtain the assertion by applying the inverse limit to the isomorphism of \cite[Lemme~2.1.3]{Bri1}. For general $K$, it follows from the equalities
\begin{align*}
\ker{(\theta_{\ok}:{\ok\otimes_{\mathcal{O}_{K_0}}\aplus [\![u_1,\dotsc,u_d]\!]\to\ocp})} & =(1\otimes\ker{\theta},\pi_K\otimes 1-1\otimes [\widetilde{\pi_K}])\\
\ker{(\theta_{\ok}:{\ok\otimes_{\mathcal{O}_{K_0}}\ainfcohen\to\ocp})} & =(1\otimes\ker{\theta},\pi_K\otimes 1-1\otimes [\widetilde{\pi_K}])
\end{align*}
and Lemma~\ref{lem3.1}.
\end{proof}

To simplify the notation, put $\A=\ainf$.

\begin{corollary}\label{gra}
\begin{enuroman}
\item There exists an isomorphism of graded $\ocp$-algebras 
\[
\mathrm{gr}_{\iplus}\A\cong\ocp [X_0,\dotsc,X_d].
\]
\item $\iplus=(\pi_K-[\widetilde{\pi_K}],t_1-[\widetilde{t_1}],\dotsc,t_d-[\widetilde{t_d}])$ where $t_1,\dotsc,t_d$ a $p$-basis of $K$.
\end{enuroman}
\end{corollary}
\begin{proof}
We have only to prove (ii). Choose a $p$-basis $t'_1,\dotsc,t'_d$ of $K_0$ such that $t'_j\equiv t_j\mod{\mathfrak{m}_{K}}$. It is easy to see $t'_j-[\widetilde{t'_j}]\equiv t_j-[\widetilde{t_j}]\mod{(\pi_K,[\widetilde{\pi_K}])\A}$. We have
\begin{align*}
t'_j-[\widetilde{t'_j}]-(t_j-[\widetilde{t_j}]) & \in\iplus\cap(\pi_K,[\widetilde{\pi_K}])\A=\iplus\cap (\pi_K-[\widetilde{\pi_K}],[\widetilde{\pi_K}])\A\\
 & = (\pi_K-[\widetilde{\pi_K}])\A+[\widetilde{\pi_K}]\iplus.
\end{align*}
Since this implies $\pi_K-[\widetilde{\pi_K}],t'_1-[\widetilde{t'_1}],\dotsc,t'_d-[\widetilde{t'_d}]$ generates $\iplus/[\widetilde{\pi_K}]\iplus$, we obtain the assertion by Nakayama's lemma.
\end{proof}

The followings are immediate:

\begin{remark}\label{rem:ainf}
\begin{enuroman}
\item The canonical maps $A_k\to A_k[p^{-1}]\to B_k$ are injective and $A_k[p^{-1}]=B_k$.
\item $\bdr$ and $B_k$ have a canonical $\kbar$-algebra structure and $B_k$ is a $p$-adic Banach algebra with a unit disc $A_k$.
\item $\iplus^k\cap I^{k+1}=\iplus^{k+1}$, $\iplus^{k+1}\cap\pi_K^n\A=\pi_K^n\iplus^{k+1}$ for $n,k\in\mathbb{N}$.
\item There exist canonical isomorphisms 
\begin{align*}
\cp\otimes_{\ocp} (\iplus^k/\iplus^{k+1}) & \cong J_k \\
\cp/\ocp\otimes_{\ocp} (\iplus^k/\iplus^{k+1}) & \cong I^k/(I^{k+1}+\iplus^k).
\end{align*}

\item There exist $($non-canonical$)$ ring isomorphisms
\begin{align*}
\bdr & \cong\cp [\![X_0,\dotsc,X_d]\!] \\
\B_k & \cong\cp [\![X_0,\dotsc,X_d]\!]/(X_0,\dotsc,X_d)^{k+1}.
\end{align*}
\end{enuroman}
\end{remark}

Since $\bdr\cong\varprojlim{B_k}$, we call the inverse limit topology of the $p$-adic Banach topology of $B_k$ the canonical topology of $\bdr$. In the following, we put $B_k$ the $p$-adic Banach topology and $\bdr$ the canonical topology if particular mention is not stated. Note that we have a canonical isomorphism $B_0\cong\cp$ as topological rings and the induced topology on $J_k$ as a subspace of $B_k$ coincides with the $p$-adic topology.

We call $V$ a $B_k$-representation of $\gl$ if $V$ is a finitely generated $B_k$-module with a continuous semilinear $\gl$-action and $V$ a $\bdr$-representation if $V$ is a $B_k$-representation for some $k\in\mathbb{N}$. $B_k$-representation $V$ is said to be $B_k$-admissible if the canonical map

\[
B_k\otimes_{B_k^{\gl}} V^{\gl}\to V
\]

is an isomorphism.

\begin{definition}
For $x\in B_k$, put $w_k(x)=\sup{\{m\in\mathbb{Z}\mid x\in\pi_K^mA_k\}}$.
\end{definition}
The following theorem will be used without citations.
\begin{theorem}[\cite{Ax}, \cite{Sen1}, \cite{Tate}]\label{ax}
For all algebraic extensions $L/K$,
\[
\widehat{L}=\cp^{\gl}.
\]
\end{theorem}

\begin{remark}
Since a $p$-adic Banach space has an orthonormal basis (\cite[Lemma~1]{Se2}), a surjection of $p$-adic Banach spaces has continuous section. In particular, we can consider the usual continuous Galois cohomology theory for $\bdr$-representations.
\end{remark}

\section{Higher K\"{a}hler differentials}\label{sec:kah}
In this section, we will study Colmez's higher K\"{a}hler differentials $\olkahler$ of an algebraic extension $L/K$. We generalize the results of \cite{Col} to the imperfect residue field case: $\S 2.1$ corresponds to $\S$A$1$, $\S 2.2$, $\S 2.3$ corresponds to $\S$A$2$ in \cite{Col}.

\subsection{Definitions and basic facts}
\begin{definition}\label{higher kahler}
Define a family of $\ok$-algebras $\{\hokbar{k} \}_{k\in \mathbb{N}}$ and $\okbar$ -modules $\{\hkahler{k} \}_{k\in \mathbb{N}_{>0}}$ as follows:

Let $\hokbar{0}=\okbar$. For $k\ge 1$, set inductively $\hkahler{k}=\okbar\otimes_{\hokbar{k-1}}{\threekahler{1}{\hokbar{k-1}}{\ok}}$ and $\hokbar{k}=\ker{d^{(k)}}$, where $d^{(k)}$ is the canonical derivation $d^{(k)}:\hokbar{k-1}\to \hkahler{k}$.

Similarly, for an algebraic extension $L/K$, put $\hol{k}=\hok{k}\cap\ol$ and $\olkahler$ the $\ol$-submodule of $\hkahler{k}$ generated by $d^{(k)}(\hol{k-1})$. We put $d^{(k)}=d^{(k)}\!\!\mid _{\hol{k-1}}$ by abuse of notation.
\end{definition}

\begin{remark}
For an algebraic extension $L/K$, we have a canonical isomorphism $\holkahler{1}\cong~\Omega_{\ol/\ok}^1$. This follows from the following lemma:
\begin{lemma}[{\cite[The footnote of p.420]{Sch2}}]\label{exactness kahler}
 For algebraic extensions $L'/L/K$,
\[\xymatrix{
0 \ar[r] & \mathcal{O}_{L'}\otimes\Omega_{\ol/\ok}^1 \ar[r] &\Omega_{\mathcal{O}_{L'}/\ok}^1 \ar[r] & \Omega_{\mathcal{O}_{L'}/\ol}^1 \ar[r] & 0 \\
}\]
is exact.
\end{lemma}
\end{remark}

\begin{theorem}[cf. {\cite[Th\'{e}or\`{e}me~1]{Col}}]\label{thm Col}
\begin{enuroman}
\item For $k\in \mathbb{N}$, we have $\hokbar{k}=\kbar\cap (\A+I^{k+1})$ and the canonical isomorphisms $\hokbar{k}/p^n\hokbar{k}\cong \ak /p^n\ak$ for $n\in \mathbb{N}$. 

\item Let $\partial^{(k)}:\hokbar{k-1} \to I^k/(I^{k+1}+\hiplus{k})$ be a $($well-defined$)$ derivation sending $x\in \hokbar{k-1}$ to $x-\tilde{x}$ with $\tilde{x}\in \A$ such that $x-\tilde{x}\in I^k$. Then we have a commutative diagram:

\[\xymatrix{ 
\hokbar{k-1} \ar[r]^{d^{(k)}} \ar[dr]_{\partial^{(k)}} & \hkahler{k} \ar@{-->}[d]^{\iota}_{\cong} \\
 & I^k/(I^{k+1}+\hiplus{k})  \\
}\]
where $\iota$ is the $\okbar$-module isomorphism induced by the universality of K\"{a}hler differentials.

\item $d^{(k)}$ is surjective.

\item $\kbar$ is dense in $\bdr$ $(resp.\ B_k)$ and $\bdr$ $(resp.\ B_k)$ is the Hausdorff completion of $\kbar$ for the topology whose fundamental neighborhood at $0$ is $\{p^n\hokbar{k}\}_{n,k\in\mathbb{N}}$ $(resp.\ \{p^n\hokbar{k}\}_{n\in\mathbb{N}})$.
\end{enuroman}
\end{theorem}

In the rest of this subsection, we give the proof of Theorem~\ref{thm Col}. The proof is identical to that of \cite{Col}, but we reproduce the proof for the convenience of the readers.

\newcommand{\kur}{K^{\mathrm{um}}}
\newcommand{\okur}{\mathcal{O}_{\kur}}
\newcommand{\hd}[1]{d^{(#1)}}
\newcommand{\oku}[1]{\mathcal{O}^{(#1)}}

For $x\in\okbar$, let $P(X)\in\ok[X]$ be the minimal polynomial of $x$ over $K$ and let $r\in\mathbb{N}$ be a natural number such that $r\ge v_p(P'(X))$. Define $r_k\in\mathbb{N}$ inductively by $r_0=0$, $r_k=3r_k+r$, i.e., $r_k=(3^k-1)r/2$. For $a,k\in\mathbb{N}$, put $r_k(a)=\inf{(r_k,v_p(a))}$ and $z_{k,a}=p^{r_k-r_k(a)}x^a$. 
\begin{lemma}\label{collem2}
For $a,k\in\mathbb{N}$, we have $z_{k,a}\in\mathcal{O}_{\kbar}^{(k)}$.
\end{lemma}
\begin{proof}
We will prove the assertion by induction on $k$. The case $k=0$ is trivial. Suppose the assertion is true for $k$. Using the relation $p^{r_k+r_k(a)}z_{k,a}=z_{k,1}(p^{r_k(a)}z_{k,a-1})$, we have the following relation by induction on $a$:
\begin{equation}\label{equcolone}
p^{r_k+r_k(a)}d^{(k+1)}(z_{k,a})=p^{r_k}ax^{a-1}d^{(k+1)}(z_{k,1}).
\end{equation}
For a polynomial $A(X)\in\ok [X]$, we have $p^{r_k}d^{(k+1)}(p^{r_k}A(x))=p^{2r_k}A'(x)d^{(k+1)}(z_{k,1})$. In particular, by putting $A(X)=P(X)$, we obtain $p^{r_k}P'(x)d^{(k+1)}(z_{k,1})=0$ and hence we have $p^{r_k+r}d^{(k+1)}(z_{k,1})=0$. Since $r_k(a)\le r_k$, we have 
\begin{equation}\label{equcoltwo}
p^{2r_k+r}d^{(k+1)}(z_{k,a})=0\ \text{for all}\ a\in\mathbb{N}
\end{equation}
by (\ref{equcolone}). Now we prove the assertion for $z_{k+1,a}$, that is, $d^{(k+1)}(z_{k+1,a})=0$. There are two cases: If $x_p(a)\le r_k$, then we have $z_{k+1,a}=p^{2r_k+r}z_{k,a}$, so we have $\hd{k+1}(z_{k+1,a})=0$ by (\ref{equcoltwo}). If $v_p(a)>r_k$, write $a=p^{r_k}b$. Then we have 
\[
\hd{k+1}(z_{k,a})=p^{r_{k+1}-r_{k+1}(a)}\hd{k+1}((z_{k,p^{r_k}})^b)=p^{r_{k+1}-r_{k+1}(a)}b(z_{k,p^{r_k}})^{b-1}\hd{k+1}(z_{k,p^{r_k}})=0
\]
since we have $v_p(b)+r_{k+1}-r_{k+1}(a)\le 2r_{k}+r$.
\end{proof}

Put $\oku{k}=\kbar\cap (\A+I^{k+1})$. We will prove the assertion (i) by induction on $k$\ : If $k=0$, there is nothing to prove. Assume $\hokbar{k-1}=\oku{k-1}=\kbar\cap (\A+I^k)$ and $A_{k-1}/p^nA_{k-1}\cong\hokbar{k-1}/p^n\hokbar{k-1}$.

\begin{lemma}\label{leminc}
We have $\hokbar{k}\subset\oku{k}$.
\end{lemma}  
\begin{proof}
By Remark~\ref{rem:ainf}(iii), it is easy to see that $\partial^{(k)}:\hokbar{k-1}\to I^k/(I^{k+1}+I_{+}^{k})$ is a well-defined  $\ok$-derivation. Then we have the $\okbar$-linear map $\iota:\threekahler{(k)}{\okbar}{\ok}\to I^k/(I^{k+1}+I_{+}^{k})$ induced by the universal property of K\"{a}hler differentials which makes the diagram in Theorem~\ref{thm Col}(ii) commutative. Since $\ker{\partial^{(k)}}=\oku{k}$, we have the conclusion. 
\end{proof}

Since the canonical map $\oku{k}\to A_k$ is injective by Remark~\ref{rem:ainf}(iii), we regard $\hokbar{k}$ and $\oku{k}$ as subrings of $A_k$.

\begin{lemma}\label{lemdense}
For all $k\in\mathbb{N}$, $\hokbar{k}$, hence $\oku{k}$, is dense in $A_k$ endowed with the $p$-adic topology. 
\end{lemma}
\begin{proof}
By Proposition~\ref{prop ainf}(ii), we have only to prove that the topological closure of $\hokbar{k}$ in $A_k$ contains the set $\{[x]|x\in\eplus\}$. Write $x=(x^{(n)})\in\eplus$ and let $P(X)\in\ok[X]$ be the  minimal polynomial of $x^{(0)}$ over $K$. For $m\in\mathbb{N}_{>0}$, let $S_m(X)=X^{p^m}+\pi_K X$ and let $x_{n,m}\in\okbar$ be an element satisfying $S_{m}(x_{n,m})=x^{(n)}$. The minimal polynomial of $x_{n,m}$ over $K$ divides $P_{n,m}(X)=P(S_{m}(X)^{p^n})$. We have $P'_{n,m}=p^nS_{m}^{'}S_{m}^{p^n-1}P'((S_m)^{p^n})$ and so $v_p(P'_{n,m}(x_{n,m}))=n+1/e_K+(1-1/p^n)v_p(x^{(0)})+v_p(P'(x^{(0)}))$ which is independent of $m$: We put this $u_n$. By Lemma~\ref{collem2}, $y_{n,m}=(x_{n,m})^{p^m}\in\hokbar{k}$ for $m\ge (3^k-1)(u_n+1)/2$. Since $\theta(y_{n,m}-[x^{p^{-n}}])=-\pi_Kx_{n,m}$, we have
\[
y_{n,m}\equiv [x^{p^{-n}}]\mod{\pi_K\A+\iplus+I^{k+1}}.
\]
Note that for a commutative ring $A$ and an ideal $\mathfrak{a}$, if $\alpha,\beta\in A$ with $\alpha\equiv \beta\mod{\mathfrak{a}}$, then we have ${\alpha}^{p^n}\equiv {\beta}^{p^n}\mod{\mathfrak{a}(p,\mathfrak{a})^n}$. Hence we have $(y_{n,m})^{p^{n}}=(x_{n,m})^{p^{n+m}}\equiv [x]\mod{\pi_K^{n-(k-1)}\A+I^{k+1}}$. Hence we can conclude that $[x]$ is in the closure of $\hokbar{k}$.
\end{proof}

\newcommand{\qok}[1]{\hokbar{#1}/p^n\hokbar{#1}}
\newcommand{\qa}[1]{A_{#1}/p^nA_{#1}}

\begin{lemma}\label{lem:thick}
For $n\in\mathbb{N}$, $\qok{k}$ and $\oku{k}/p^n\oku{k}$ are infinitesimal thickenings of $\ocp/p^n\ocp$ over $\ok/p^n\ok$ of order $k$.
\end{lemma}
\begin{proof}
First, consider $\qok{k}$. By Lemma~\ref{lemdense}, the canonical map $\qok{k}\to\qa{k}$ is surjective, so the composition map $\theta:\qok{k}\to\qa{k}\to\ocp/p^n\ocp$ is surjective. We have only to prove $(\ker{\theta})^{k+1}=0$. Let $\theta_1:\qok{k}\to\qok{k-1}\cong\qa{k-1}$ be the canonical projection and let $\theta_2:\qa{k-1}\to\ocp/p^n\ocp$ be the map induced by $\theta :A_{k-1}\to \ocp$, we have $\theta=\theta_2\circ\theta_1$. Since $\qa{k-1}$ is an infinitesimal thickening of $\ocp/p^n\ocp$ over $\ok/p^n\ok$ of order $(k-1)$, we have $(\ker{\theta})^{k}\subset \ker{\theta_1}$. Thus, we have only to prove that if $\bar{x}\in\ker{\theta},\bar{y}\in\ker{\theta_1}$, then $\bar{x}\bar{y}=0$. By definition, we have $x\in\hokbar{k}\cap p^n\okbar$ and $y\in\hokbar{k}\cap p^n\hokbar{k-1}$. Since we have $p^n\hd{k}(p^{-n}y)=0$ and $x\in p^n\okbar$, we have $\hd{k}(xp^{-n}y)=x\hd{k}(p^{-n}y)=0$, which implies $xy\in p^n\hokbar{k}$, i.e., $\bar{x}\bar{y}=0$.

For $\oku{k}/p^n\oku{k}$, the same proof as above works if we replace $d^{(k)}$ by $\partial^{(k)}$.
\end{proof}

By the universality of infinitesimal thickening (\cite[$\S 1.1$]{Fon1}), we have

\begin{corollary}\label{cor:thick}
$\qa{k}$, $\qok{k}$, $\oku{k}/p^n\oku{k}$ are isomorphic to each other.
\end{corollary}
\begin{proof}
First, we prove the canonical map $\oku{k}/p^n\oku{k}\to\qa{k}$ is an isomorphism. The surjectivity follows by Lemma~\ref{lemdense} and the injectivity is from
\[
\mathcal{O}^{(k)}\cap (p^n\A+I^{k+1})=\kbar\cap p^n(\A+I^{k+1})=p^n(\kbar\cap (\A+I^{k+1}))=p^n\mathcal{O}^{(k)}.
\]
Let $\beta_n:\qok{k}\to\qa{k}$ be the canonical map. By Lemma~\ref{lem:thick} and the universality of infinitesimal thickenings, we have a canonical $\ok /p^n\ok$-homomorphism $\alpha_n:\qa{k}\to\qok{k}$ with $\beta_n\circ\alpha_n =\mathrm{id}$. Let $\widehat{\mathcal{O}_{\kbar}^{(k)}}=\varprojlim{{\mathcal{O}}_{\kbar}^{(k)} /p^n\mathcal{O}_{\kbar}^{(k)}}$. Since $\alpha_n$, $\beta_n$ are compatible with the inverse systems, there exists $K$-algebra homomorphisms $\alpha:A_k[p^{-1}]\to\widehat{{\mathcal{O}}_{\kbar}^{(k)}}[p^{-1}]$, $\beta:\widehat{{\mathcal{O}}_{\kbar}^{(k)}}[p^{-1}]\to A_k[p^{-1}]$ with $\beta\circ\alpha=\mathrm{id}$. To prove that $\alpha_n$, $\beta_n$ are inverse to each other, we have only to prove $\alpha$, $\beta$ are inverse to each other. Since the kernel of the canonical projections $\theta:A_k/p^nA_k\to\ocp/p^n\ocp$, $\theta:{\mathcal{O}}_{\kbar}^{(k)}/p^n{\mathcal{O}}_{\kbar}^{(k)}\to\ocp/p^n\ocp$ are nilpotent of exponent $k+1$, $A_k[p^{-1}]$ and ${\widehat{\mathcal{O}}}_{\kbar}^{(k)}[p^{-1}]$ are local rings with the same residue field $\cp$, whose maximal ideals are nilpotent of exponent $k+1$. By the ind-\'{e}taleness of $\kbar/K$, the $\kbar$-algebra structure of $\cp$ uniquely lifts to $\kbar$-algebra structures of $\widehat{{\mathcal{O}}_{\kbar}^{(k)}}[p^{-1}]$ and $A_k[p^{-1}]$ and $\alpha |_{\kbar}$, $\beta |_{\kbar}$ are inverse to each other. On the other hand, $\kbar={\mathcal{O}}_{\kbar}^{(k)}[p^{-1}]$ is a dense subring of $\widehat{{\mathcal{O}}_{\kbar}^{(k)}}[p^{-1}]$ and $\kbar$ is also dense in $A_{k}[p^{-1}]$ by the surjectivity of $\beta$. This denseness of $\kbar$ implies that $\alpha$, $\beta$ are inverse to each other.
\end{proof}

\begin{corollary}\label{corequ}
$\hokbar{k}=\oku{k}$.
\end{corollary}
\begin{proof}
We have $\hokbar{k}\subset\oku{k}$ and $\hokbar{k}/p\hokbar{k}\cong\oku{k}/p\oku{k}$. This implies that multiplication by $p$ on $\oku{k}/\hokbar{k}$, whose elements are killed by some powers of $p$, is an isomorphism. So we have the conclusion.
\end{proof}

By the above corollaries, we obtain (i).

Next, let us prove (ii). Well-definedness of $\partial^{(k)}$ is proved in the proof of Lemma~\ref{leminc}. Since we have Corollary~\ref{corequ}, we have only to prove 
\begin{lemma}\label{lem parsurj}
$\partial^{(k)}$ and $d^{(k)}$ are surjective.
\end{lemma}
\begin{proof}
First, we prove that $\partial^{(k)}$ is surjective. By an exact sequence
\[\xymatrix{
0\ar[r]& \oku{k}\ar[r]&\hokbar{k-1}\ar[r]&\mathrm{Im}\partial^{(k)}\ar[r]&0,\\
}\]
we can deduce an exact sequence
\[\xymatrix{
0\ar[r]&(\mathrm{Im}\partial^{(k)})[p^n]\ar[r]&\qa{k}\ar[r]&\qa{k-1}\ar[r]&0.\\
}\]
By the properties of $\A$, we have
\begin{align*}
(\mathrm{Im}\partial^{(k)})[p^n]\subset I^k/(\iplus^k+I^{k+1})[p^n] & \cong (\cp/\ocp\otimes_{\ocp}\iplus^{k}/\iplus^{k+1})[p^n]\cong\iplus^{k}/(p^n\iplus^k+\iplus^{k+1})\\
 & \cong\ker{(A_k/p^nA_k\to A_{k-1}/p^nA_{k-1})}\cong (\mathrm{Im}\partial^{(k)})[p^n]
\end{align*}
where the composition of these maps is identity. Hence we have $(\mathrm{Im}\partial^{(k)})[p^n]=I^k/(\iplus^k+I^{k+1})[p^n]$ and this implies the conclusion since $I^k/(\iplus^k+I^{k+1})$ is $p$-torsion.

Since $\mathrm{Im}d^{(k)}\cong\mathrm{Im}\partial^{(k)}=I^k/(I^k_{+}+I^{k+1})$ is an $\okbar$-module and generates $\threekahler{(k)}{\okbar}{\ok}$ over $\okbar$, $d^{(k)}$ is also surjective.  
\end{proof}

(iii) is already proved. (iv) is a direct consequence of (i).

\begin{remark}\label{rem:kah}
In the following, we canonically identify $\threekahler{(k)}{\okbar}{\ok}=I^{k}/(I^{k+1}+I_{+}^{k})=\cp/\ocp\otimes (I_{+}^k/I_{+}^{k+1})$, $V_p(\threekahler{(k)}{\okbar}{\ok})= J_k$ by Remark~\ref{rem:ainf} and Theorem~\ref{thm Col}. Note that $V_p(\threekahler{(k)}{\ol}{\ok})\subset J_k^{\gl}$.
\end{remark}

\subsection{Expansion in a power series ring}\label{min}
Let $K$ be a field of characteristic 0, $d\in\mathbb{N}$ and let $\mathfrak{m} \subset K[X_0,\dotsc,X_d]$ be a maximal ideal. Put $\km=\varprojlim_{k\in \mathbb{N}} K[X_0,\dotsc,X_d]/\mathfrak{m}^{k+1}$, which is a $(d+1)$-dimensional complete regular local ring, and let we put $L=K[X_0,\dotsc,X_d]/\mathfrak{m}$, $\pi_i=\mathrm{pr}(X_i)\in L$.

Choose a regular parameter $P_0,...,P_d\in \mathfrak{m} \km$  and a subset $S\subset\mathbb{N}^{d+1}$ such that $0\in S$ and $L=\oplus_{n\in S} K\pi^n$ and put $\Lambda_S=\oplus_{n\in S} KX^n\subset K[X_0,\dotsc,X_d]$.

\begin{definition}
Regard $L$ as a sub-$K$-algebra of $\km$ by the canonical $K$-embedding $L\to~\km$. For $x\in L$, we define $\lambda_{n,x,P}(X)\in~\Lambda_S\ (n\in\mathbb{N}^{d+1})$ by the expansion $x=\sum_{n\in\mathbb{N}^{d+1}}{\lambda_{n,x,P}(X)P^n(X)}$ in $K[X]_{\mathfrak{m}}$.
\end{definition}

\subsection{The fundamental properties of higher K\"{a}hler differentials}\label{fund property}
In this subsection, we will prove the following theorem:
\begin{theorem}\label{fund kah}
Let $k\in\mathbb{N}_{>0}$.
\begin{enuroman}
\item Let $L/K$ be a finite extension of local fields.
Then $\mu(\olkahler)\le \binom{d+k}{k}$.

\item For finite extensions $L_2/L_1/K$, the canonical morphism $\intring{L_2}\otimes \threekahler{(k)}{\mathcal{O}_{L_1}}{\mathcal{O}_{K}}\to \threekahler{(k)}{\mathcal{O}_{L_2}}{\mathcal{O}_{K}}$ is injective.
\end{enuroman}
\end{theorem}

\newcommand{\kmp}[1]{k_K^{p^{-#1}}}
\newcommand{\kp}{k_K^p}
\newcommand{\lp}{k_L^p}
\newcommand{\lmp}[1]{k_L^{p^{-#1}}}
\newcommand{\pibar}[1]{\overline{\pi_{#1}}}
\newcommand{\tbar}[1]{\overline{t_{#1}}}
\newcommand{\fbar}[1]{\overline{f_{#1}}}
\newcommand{\kkmp}{\mathbf{k}^{p^{-1}}}

For a while, we assume that a finite extension $L/K$ satisfies the following hypothesis:

\begin{hyp}\label{hyp}
$L/K$ has no unramified subextensions over $K$.
\end{hyp}

Note that this hypothesis is equivalent to the condition that $k_K$ is separably closed in $k_L$.
\begin{lemma}
Let $\mathbf{k}$ be a field of characteristic $p$ and let $\mathbf{l}/\mathbf{k}$ be a finite purely inseparable extension. Assume $\mathbf{l}\subset\kkmp$. Then the first fundamental sequence
\[\xymatrix{
0\ar[r]& \kkmp\otimes\threekahler{1}{\mathbf{l}}{\mathbf{k}} \ar[r]&\threekahler{1}{\kkmp}{\mathbf{k}}\ar[r] & \threekahler{1}{\kkmp}{\mathbf{l}}\ar[r] & 0 \\
}\]
is exact.
\end{lemma}
\begin{proof}
We have only to prove the injectivity. Choose a $p$-basis $\{x_{\lambda}\}_{\lambda\in\Lambda}$ of $\mathbf{l}/\mathbf{k}$ and $p$-basis $\{y_{\gamma}\}_{\gamma\in\Gamma}$ of $\kkmp/\mathbf{l}$. Then, we have
\[
\mathbf{l}=\oplus_{n\in\mathbb{N}^{\oplus\Lambda},0\le n_{\lambda}<p}{\mathbf{k}x^n},\ \kkmp=\oplus_{m\in\mathbb{N}^{\oplus\Gamma},0\le m_{\gamma}<p}{\mathbf{l}y^m}.
\] 
So we have $\kkmp=\oplus_{n\in\mathbb{N}^{\oplus\Lambda},m\in\mathbb{N}^{\oplus\Gamma},0\le n_{\lambda},m_{\gamma}<p}{\mathbf{k}x^ny^m}$. This implies that $\{x_{\lambda}\}\cup \{y_{\gamma}\}$ forms a $p$-basis of $\kkmp/\mathbf{k}$. In particular, $\{x_{\lambda}\}$ is $p$-independent in $\mathbf{k}^{p^{-1}}/\mathbf{k}$. 
\end{proof}

\begin{corollary}
If $\dim_{k_L}{\threekahler{1}{k_L}{k_K}}=d$, then $\kmp{1}\subset k_L$.
\end{corollary}
\begin{proof}
Since, by the lemma,
\[\xymatrix{
0 \ar[r] & k_L\otimes\threekahler{1}{k_K\lp}{\lp}\ar[r] &\threekahler{1}{k_L}{\lp} \ar[r] & \threekahler{1}{k_L}{k_K\lp}\ar[r] & 0\\ 
}\]
is exact and $\threekahler{1}{k_L}{k_K}=\threekahler{1}{k_L}{k_K\lp}$, we have $\threekahler{1}{k_K\lp}{\lp}=0$. Hence we have $k_K\lp=\lp$, i.e., $k_K\subset\lp$.
\end{proof}

Put $d'=\dim_{k_L}{\threekahler{1}{k_L}{k_K}}$. Then we can choose $\overline{t_{d'+1}},\dotsc,\overline{t_d}\in k_K$ such that $\dim_{k_L}{\threekahler{1}{k_L(\overline{t_{d'+1}}^{p^{-1}},\dotsc,\overline{t_{d}}^{p^{-1}})}{k_K}}=d$: In fact, by the lemma,
\[\xymatrix{
0\ar[r]&\kmp{1}k_L\otimes\threekahler{1}{k_L}{k_K\lp}\ar[r]&\threekahler{1}{\kmp{1} k_L}{(\kmp{1} k_L)^p}\ar[r]&\threekahler{1}{\kmp{1}k_L}{k_L}\ar[r]&0
}\]
is exact and $\threekahler{1}{k_L}{k_K}=\threekahler{1}{k_L}{k_K\lp}$. Since the canonical map $\kmp{1}k_L\otimes\threekahler{1}{\kmp{1}}{k_K}\to\threekahler{1}{\kmp{1}k_L}{k_L}$ is surjective, we can choose $\overline{t_{d'+1}},\dotsc,\overline{t_d}\in k_K$ such that $d\overline{t_{d'+1}}^{p^{-1}},\dotsc,d\overline{t_{d}}^{p^{-1}}$ forms a basis of $\threekahler{1}{\kmp{1}k_L}{k_L}$ and these elements satisfy the required property. 

Since $\overline{t_{d'+1}},\dotsc,\overline{t_d}$ is $p$-independent over $k_K$ by the above argument, choose $\overline{t_1},\dotsc,\overline{t_d}\in k_K$ such that $\overline{t_1},\dotsc,\overline{t_{d'}},\overline{t_{d'+1}},\dotsc,\overline{t_d}$ forms a $p$-basis of $k_K$. If we take a $p$-basis $\pibar{1},\dotsc,\pibar{d'}$ of $k_L/k_K$, then by the corollary, $k_K\subset (k_L(\tbar{d'+1}^{p^{-1}},\dotsc,\tbar{d}^{p^{-1}}))^p=\kp(\tbar{d'+1},\dotsc,\tbar{d})(\pibar{1}^p,\dotsc,\pibar{d'}^p)$ since $k_L=k_K(\pibar{1},\dotsc,\pibar{d'})$. In particular, we can choose $\fbar{j}\in\kp [X_1^p,\dotsc,X_{d'}^p,T_{d'+1},\dotsc,T_d]$ for  $1\le j \le d'$ with $\tbar{j}=\fbar{j}(\pibar{1},\dotsc,\pibar{d'},\tbar{d'+1},\dotsc,\tbar{d})$. Choose $r_1,\dotsc,r_{d'}\in\mathbb{N}$ such that $k_L=\oplus_{0\le n_j< p^{r_j}}{k_K\pibar{1}^{n_1}\dotsb \pibar{d'}^{n_{d'}}}$.

Let $t_j\in\ok$ for $0\le j\le d$ be a lift of $\tbar{j}\in k_K$, which forms a $p$-basis of $K$. Let $f_j\in\ok [X_1^p,\dotsc,X_{d'}^p,T_{d'+1},$
\noindent $\dotsc,T_{d}]$ for $1\le j\le d'$ be a lift of $\fbar{j}\in \kp [X_1^p,\dotsc,X_{d'}^p,T_{d'+1},\dotsc,T_{d}]$. Let $\pi_j\in\ol$ and $e_j\in\mathbb{N}$ for $0\le j\le d$ as 

\[
\pi_j =
\begin{cases}
\pi_L & j=0\ \text{and}\ e_{L/K}>1  \\
\pi_K & j=0\ \text{and}\ e_{L/K}=1  \\
\text{a lift of}\ \pibar{j} & 1\le j\le d' \\
t_j & j>d' 
\end{cases}
,\quad e_j = 
\begin{cases}
e_{L/K} & j=0\\
p^{r_j} & 1\le j\le d' \\
1 & j>d' 
\end{cases}
\]
and put $\Lambda=\oplus_{0\le n_j< e_j}{\ok X_0^{n_0}\dotsb X_{d'}^{n_{d'}}}\subset\ok [X_0,\dotsc,X_{d'}]$.

We use the following lemma in the construction below:
\begin{lemma}\label{lemstr}
$L=K(\pi_0,\pi_1,\dotsc,\pi_{d'})$ and $\ol=\oplus_{0\le n_j <e_j}{\ok\pi_0^{n_0}\dotsb \pi_{d'}^{n_{d'}}}$. For $x=\sum_{0\le n_j< e_j}{a_n\pi_0^{n_0}\dotsb \pi_{d'}^{n_{d'}}}\in\ol$ with $a_n\in\ok$,
\[
v_K(x)=\inf_{0\le n_j <e_j}{v_K(a_n\pi_0^{n_0})}.
\]
\end{lemma}
\begin{proof}
The first part is \cite[II, Proposition~2.4]{Fes}. The latter part follows from
\[
\pi_L\Bigm|\sum_{0\le n_j< e_j}{a_{n_0n_1\dotsb n_{d'}}\pi_1^{n_1}\dotsb\pi_{d'}^{n_{d'}}}\Longleftrightarrow \pi_K \bigm| a_{n_0n_1\dotsb n_{d'}}\ \text{for all }n_1,\dotsc,n_{d'} 
\]
and $v_K(\pi_0^{n_0})=n_0/e_0\in\left\{0,1/e_0,\dotsc,e_0-1/e_0\right\}$.
\end{proof}

Finally, put $P_j\in\ok [X_0,\dotsc,X_d]$ for $0\le j \le d$ as follows:

$\bullet\ j=0$: There exists a unique $g_0\in\Lambda$ such that $\pi_0^{e_0}=g_0(\pi)$. Put $P_0=X_0^{e_0}-g_0$. Note that $P_0\equiv X_0^{e_0}\mod{\pi_K\ok [X_0,\dotsc,X_d]}$ by Lemma~\ref{lemstr}.

$\bullet\ 1\le j\le d'$: There exists a unique $g_j\in\Lambda$ such that $f_j(\pi_1,\dotsc,\pi_{d'},t_{d'+1},\dotsc,t_{d})-t_j=g_j(\pi)$. Put $P_j=f_j(X_1,\dotsc,X_{d'},t_{d'+1},\dotsc,t_{d})-g_j(X_0,\dotsc,X_{d'})-t_j$. Note that $X_0|\ g_j(X)\mod{\pi_K\ok [X_0,\dotsc,X_d]}$, again by Lemma~\ref{lemstr}.

$\bullet\ j>d'$: Put $P_j=X_j-t_j$.

\newcommand{\til}[1]{\widetilde{#1}}
\newcommand{\bra}[1]{[#1]}
\newcommand{\bt}[1]{[\widetilde{#1}]}
\newcommand{\dbt}[1]{#1-[\widetilde{#1}]}

Let $\Pi_j$ be an element of $\A$ such that $\theta(\Pi_0)=\pi_0$ with $\Pi_0\in\aplus [\pi_K]$ if $e_{L/K}>1$, $\Pi_0=[\widetilde{\pi_K}]$ if $e_{L/K}=1$, $\theta(\Pi_j)=\pi_j$ for $1\le j\le d'$ (for example, $\Pi_j=[\til{\pi_j}]$) and $\Pi_j=[\til{t_j}]$ for $j>d'$. Note that if $k_K$ is perfect, the condition about $\Pi_0$ is not necessary since $\A=\aplus [\pi_K]$. 
\begin{proposition}\label{gen}
$P_0(\Pi),\dotsc,P_d(\Pi)$ is a generator of $I_{+}$.
\end{proposition}  
\begin{proof}
Obviously $P_0(\Pi),\dotsc,P_d(\Pi)\in\iplus$, so we have only to prove that $P_0(\Pi),\dotsc,P_d(\Pi)$ generates $\iplus/(\iplus,\pi_K)\iplus$ by Nakayama's lemma. We identify $\iplus/(\iplus,\pi_K)\iplus$ with $(\iplus+\pi_K\A)/(\iplus^2+\pi_K\A)$ which is a finite free $\ocp/\pi_K\ocp$-module with a basis $\{[\til{\pi_K}],t_1-\bra{\til{t_1}},\dotsc,t_d-\bra{\til{t_d}}\}$. We will prove the assertion by calculating the coefficient matrix $C$ of $\{P_0(\Pi),\dotsc,P_d(\Pi)\}$ with respect to this basis.

Write $\Pi_0=\bra{\til{\pi_L}}+\varepsilon_0$, $\varepsilon_0\in\iplus$. Since $v_{\mathbb{E}}(\til{\pi_L}^{e_0})=v_{\mathbb{E}}(\til{\pi_K})$, we can write $\til{\pi_L}^{e_0}=\til{\pi_K}\til{u}$ for some $u\in\ocp^{\times}$. If $e_{L/K}>1$, we have 
\begin{align*}
P_0(\Pi) & \equiv\Pi_0^{e_0}\mod{\pi_K\A}\\
 & \equiv\bra{\til{\pi_L}}+e_0\bra{\til{\pi_L}}^{e_0-1}\varepsilon_0\mod{\iplus^2+\pi_K\A}\\
 & \equiv\bra{\til{u}}\bra{\til{\pi_K}}+e_0\bra{\til{\pi_L}}^{e_0-1}\varepsilon_0\equiv u\bra{\til{\pi_K}}+e_0\pi_L^{e_0-1}\varepsilon_0. \\
\end{align*}
If $e_{L/K}=1$, $P_0(\Pi)\equiv\bra{\til{\pi_K}}$.

We use the following sublemma for the calculation of $P_j(\Pi)$ for $1\le j\le d'$.

\begin{sublemma}
Let $f(X_1,\dotsc,X_n)\in\okbar[X_1,\dotsc,X_n]$ and $x_1,\dotsc,x_n\in\okbar$. Then 
\[
f(x_1,\dotsc,x_n)\equiv f(\bra{\til{x_1}},\dotsc,\bra{\til{x_n}})+\displaystyle\sum_{1\le j\le n}{\frac{\partial f}{\partial X_j}(x_1,\dotsc,x_n)(x_j-\bra{\til{x_j}})}\mod{\iplus^2}.
\]
Moreover, suppose we are given $\pi\in\okbar$ with $0<v_p(\pi)\le 1$ and $b_m\in\okbar\ (m\in\mathbb{N})$ of which all but finitely many is zero, such that $f\equiv\sum_{m}{b_m^pX^m}\mod{\pi\okbar[X_1,\dotsc,X_n]}$ and put $\til{f}=\sum_{m}{\bra{\til{b_m}}^pX^m}$. Then we have 
\[
f(x_1,\dotsc,x_n)\equiv\til{f}(\bt{x_1},\dotsc,\bt{x_n})+\sum_{1\le j\le n}{\frac{\partial f}{\partial X_j}(x_1,\dotsc,x_n)(x_j-\bt{x_j})}\mod{\iplus^2+\pi\A}.
\]
\end{sublemma}

Let us prove the sublemma. Since $x-\bt{x}\in\iplus$ for $x\in\okbar$, the first assertion is just Taylor expansion. Since $x^p\equiv\bt{x}^p+px^{p-1}(x-\bt{x})\mod{\iplus^2}$ for $x\in\okbar$, we have $f(\bt{x_1},\dotsc,\bt{x_n})\equiv\til{f}(\bt{x_1},\dotsc,\bt{x_n})$ $\mod{\iplus^2+\pi\A}$, which proves the latter assertion.

Since $\Pi^p_j\equiv\bt{\pi_j}^p\mod{\iplus^2+p\A}$, we have for $1\le j\le d'$,\[
f_j(\Pi_1,\dotsc,\Pi_{d'},t_{d'+1},\dotsc,t_d)\equiv f_j(\bt{\pi_1},\dotsc,\bt{\pi_{d'}},t_{d'+1},\dotsc,t_{d})\mod{\iplus^2+\pi_K\A}.
\]
Applying the sublemma to $f_j(X_1,\dotsc,X_{d'},t_{d'+1},\dotsc,t_{d})\in\okbar[X_1,\dotsc,X_{d'}]$, we see that the RHS is equivalent to $f_j(\pi_1,\dotsc,\pi_{d'},t_{d'+1},\dotsc,t_{d})$ modulo $\iplus^2+\pi_K\A$. Again, applying the sublemma to $f_j(X_1,\dotsc,X_{d'},T_{d'+1},\dotsc,T_d)\in\okbar[X_1,\dotsc,X_{d'},T_{d'+1},\dotsc,T_d]$ with $\pi=\pi_K$, we see that it is equivalent to
\[
\til{f_j}(\bt{\pi_1},\dotsc,\bt{\pi_{d'}},\bt{t_{d'+1}},\dotsc,\bt{t_d})+\sum_{d'+1\le k\le d}{\frac{\partial f_j}{\partial T_k}}(\bt{\pi_1},\dotsc,\bt{\pi_{d'}},\bt{t_{d'+1}},\dotsc,\bt{t_{d}})(t_k-\bt{t_k})
\]
modulo $\iplus^2+\pi_K\A$. Hence we have 
\[
P_j(\Pi)\equiv\til{f_j}(\bt{\pi_1},\dotsc,\bt{\pi_{d'}},\bt{t_{d'+1}},\dotsc,\bt{t_d})-g_j(\Pi_0,\dotsc,\Pi_{d'}) -t_j\mod{\iplus^2+(\pi_K,\dbt{t_{d'+1}},\dotsc,\dbt{t_d})\A}.
\]
By the definition of $f_j$, we have $\til{f_j}(\bt{\pi_1},\dotsc,\bt{\pi_{d'}},\bt{t_{d'+1}},\dotsc,\bt{t_d})-\bt{t_j}\in (p,\bt{\pi})\aplus$ for some $\pi\in\mathfrak{m}_{\cp}$. Let $\mathfrak{j}$ be the ideal of $\A$ generated by $\{\bt{x}|x\in\mathfrak{m}_{\cp}\}$, then
\[
P_j(\Pi)\equiv\bt{t_j}-t_j-g_j(\Pi)\mod{\iplus^2+(\pi_K,\dbt{t_{d'+1}},\dotsc,\dbt{t_d}})\A+\mathfrak{j}.
\]
Note that we have $\aplus[\pi_K]\cap\iplus=(\dbt{p},\dbt{\pi_K})\aplus[\pi_K]\subset \pi_K\A +\mathfrak{j}$: Indeed, when $x=f(\pi_K)\in\iplus$ with $f\in\aplus[X]$, $x=f(\pi_K)-f([\widetilde{\pi_K}])+f([\widetilde{\pi_K}])\in (\pi_K-[\widetilde{\pi_K}])\aplus [\pi_K]+\aplus\cap I_{+}\subset (\pi_K-[\widetilde{\pi_K}],p-[\widetilde{p}])\aplus [\pi_K]$. By the definition of $g_j$, $g_j(\Pi_0,\dotsc,\Pi_{d'})\in (\pi_K,\Pi_0)\A\subset\pi_K\A+\mathfrak{j}$ since $\Pi_0-\bt{\pi_L}=\varepsilon_0\in\aplus[\pi_K]\cap\iplus\subset\pi_K\A+\mathfrak{j}$. We have $(\pi_K,\dbt{t_{d'+1}},\dotsc,\dbt{t_d})\A+\mathfrak{j}=(\dbt{\pi_K},\dbt{t_{d'+1}},\dotsc,\dbt{t_d})\A+\mathfrak{j}$ and $P_j(\Pi)\equiv \bt{t_j}-t_j\mod{\iplus^2}+(\dbt{\pi_K},\dbt{t_{d'+1}},\dotsc,\dbt{t_{d}})\A+\mathfrak{j}$. Hence we can conclude 
\begin{align*}
P_j(\Pi)-(\bt{t_j}-t_j) & \in\iplus\cap \left(\iplus^2+(\dbt{\pi_K},\dbt{t_{d'+1}},\dotsc,\dbt{t_d})\A+\mathfrak{j}\right)\\
 & =\iplus^2+(\dbt{\pi_K},\dbt{t_{d'+1}},\dotsc,\dbt{t_d})\A+\mathfrak{j}\iplus,
\end{align*}
since $\mathfrak{j}\cap\iplus=(\bigcup_{x\in\mathfrak{m}_{\cp}} [\widetilde{x}]\A)\cap I_{+}=\bigcup_{x\in\mathfrak{m}_{\cp}} [\widetilde{x}]I_{+}=\mathfrak{j}\iplus$.

For $j>d'$, we have $P_j(\Pi)=\bt{t_j}-t_j$.

Therefore, the coefficient matrix $C\in \mathrm{M}_{d+1}(\ocp/\pi_K\ocp)$ of $\{P_0(\Pi),\dotsc,P_d(\Pi)\}$ for a basis $\{\bt{\pi_K},\dbt{t_1},\dotsc,\dbt{t_d}\}$ satisfies
\[C\equiv
\left(
\begin{array}{ccc}
 \text{unit} & * & 0 \\
 0 & -E_{d',d'} & 0 \\
0& *&-E_{d-d',d-d'}\\
\end{array}
\right){\mod{\mathfrak{m}_{\cp}}\mathrm{M}_{d+1}(\ocp/\pi_K\ocp)}.
\]
In particular, $C$ is invertible, hence $P_0(\Pi),\dotsc,P_d(\Pi)$ generates $(\iplus+\pi_K\A)/(\iplus^2+\pi_K\A)$.
\end{proof}

Let $\mathfrak{m}$ be the kernel of a surjection of $K$-algebras $K[X_0,\dotsc,X_d]\to L;\ X_j\mapsto \pi_j$. Use the notation of the previous subsection by identifying $L$ with $K[X]/\mathfrak{m}$. Let $f$ be the $K$-algebra homomorphism of equi-dimensional regular local rings
\[
f:K[X]_{\mathfrak{m}}\to\bdr ;\ X_j\longmapsto\Pi_j.
\]
Since $\cp\otimes\mathrm{gr}^{1}f: \cp\otimes\mathrm{gr}^{1}K[X]_{\mathfrak{m}}\to\mathrm{gr}^{1}\bdr$ is surjective, $P_0(X),\dotsc,P_d(X)$ is a regular parameter of $K[X]_{\mathfrak{m}}$. So, we use the notation of the previous subsection with $\Lambda_S=\Lambda$.

\begin{definition}
For $x\in \bdr$, put $s_k(x)=\mathrm{sup}\{m\in \mathbb{Z}\mid x\in \pi_{K}^{m}\A+I^{k+1}\}=w_k(\overline{x})\in \mathbb{Z}\,\cup\{\infty\}$. Let $Q_0,\dotsc,Q_d$ be a generator of $I_+$. A minimal expansion $($\'{e}criture minimale in \cite{Col}$)$ of $x$ in $\bdr$ is an expansion $x=\sum_{n\in \mathbb{N}^{d+1}}{\lambda_n Q^n}$ with $\lambda_n Q^n \in \pi_{K}^{s_{|n|}(x)}\hiplus{|n|}$.
\end{definition}
Note that $s_0(x)\ge s_1(x)\ge\dotsb \ge s_k(x)\ge\cdots$ for $x\in\bdr$ and $s_0(x)=[v_K(x)]$ for $x\in\kbar$.

\begin{proposition}
For $x\in L$, $x=\sum_{n\in\mathbb{N}^{d+1}}{\lambda_{n,x,P}(\Pi)P^n(\Pi)}$ is a minimal expansion of $x$ in $\bdr$.
\end{proposition}

\newcommand{\absol}[1]{\mid{#1}\mid}
\newcommand{\aabsol}[1]{\mid\!{#1}\!\mid}
\newcommand{\aaabsol}{\mid\!\! n\!\!\mid}

\begin{proof}
For a polynomial $\lambda(X)=\sum{a_nX^n}\in K[X_0,\dotsc,X_d]$, put $v_{K}^{G}(\lambda(X))=\inf{v_K(a_n)}$. Note that, for $\lambda(X)\in\Lambda$, we have $v_K^G(\lambda(X))=[v_K(\lambda(\pi))]$ by Lemma~\ref{lemstr}, where $[*]$ denotes the integer part of $*$. We claim that $v_K^{G}(\lambda_{n,x,P}(X))\ge s_{\absol{n}}(x)$ for all $n$. By the definition of $v_K^{G}$, this claim proves the assertion. Let us prove this claim by induction on $k=\aaabsol$: When $k=0$, $\lambda_{0}(X)\in\Lambda$ satisfies $x=\lambda_{0,x,P}(\pi)$. Hence we have $s_0(x)=[v_K(x)]=v_K^G(\lambda_{0,x,P}(X))$. For $k>0$, by multiplying a power of $\pi_K$, we can assume $s_k(x)=0$. Since $s_0(x)\ge s_1(x)\ge\dotsb\ge s_{k}(x)\ge 0$ and induction hypothesis, we have $\lambda_{n,x,P}(\Pi)\in \A$ for $\aabsol{n}< k$ and $x\in\hokbar{k}=\kbar\cap (\A+I^{k+1})$. Hence we have 
\[
\sum_{\absol{n}=k}{\lambda_{n,x,P}(\Pi)P^n(\Pi)}=x-\sum_{\absol{n}<k}{\lambda_{n,x,P}(\Pi)P^n(\Pi)}-\sum_{\absol{n}>k}{\lambda_{n,x,P}(\Pi)P^n(\Pi)}\in I^k\cap (\A+I^{k+1})=\iplus^k+I^{k+1}
\]
where the last equality is from $\S 1.2$. Once we identify $\iplus^k/\iplus^{k+1}$ as a free $\ocp$-module with a basis $\{P^n(\Pi)\}_{\absol{n}=k}$, we have $I^k/(\iplus^k+I^{k+1})=\cp/\ocp\otimes_{\ocp}\iplus^{k}/\iplus^{k+1}=\oplus_{\absol{n}=k}(\cp/\ocp)P^n(\Pi)$ by Remark~\ref{rem:ainf}(iv). Thus we have $\lambda_{n,x,P}(\pi)\in\ocp$ for $|n|=k$ and $v_K^G(\lambda_{n,x,P}(X))=[v_K(\lambda_{n,x,P}(\pi))]\ge 0=s_{\absol{n}}(x)$.
\end{proof}

\begin{corollary}
Under the canonical identification $\threekahler{(k)}{\okbar}{\ok}=I^k/(\iplus^k+I^{k+1}) =\cp/\ocp\otimes (I_{+}^k/I_{+}^{k+1})$, for $x\in\hol{k-1}$, we have
\[
d^{(k)}(x)=\sum_{\mid n\mid=k}{\lambda_{n,x,P}(\pi)\otimes P^n(\Pi)}.
\]
\end{corollary}
\begin{proof}
As we see in the above proof, for $x\in\hol{k-1}$, i.e., $s_{k-1}(x)\ge 0$, we have $\widetilde{x}=\sum_{\absol{n}\le k-1}{\lambda_{n,x,P}(\Pi)P^n(\Pi)}\in\A$. Hence we have 
\begin{align*}
d^{(k)}(x)=\partial^{(k)}(x) & \equiv x-\widetilde{x}\mod{I^{k+1}+\iplus^k}\\
 & \equiv\sum_{\absol{n}=k}{\lambda_{n,x,P}(\Pi)P^n(\Pi)}\equiv\sum_{\absol{n}=k}{\lambda_{n,x,P}(\pi)\otimes P^n(\Pi)}.
\end{align*}\end{proof}

\begin{proof}[Proof of Theorem~\ref{fund kah}]
Since we have $\hokbar{k}=\mathcal{O}_{\overline{K'}}^{(k)}$ and a canonical isomorphism $\threekahler{(k)}{\okbar}{\ok}\cong\threekahler{(k)}{{\mathcal{O}}_{\overline{K'}}}{\mathcal{O}_{K'}}$ for an unramified finite extension $K'/K$, we can assume that $L/K$ satisfies Hypothesis~\ref{hyp} by replacing $K$ by its maximal unramified extension in $L$. So, we use the notation as above. First prove (ii). To prove this, we have only to prove that the canonical map $\okbar\otimes\holkahler{k}\to \hkahler{k}$ is injective. We identify $\iplus^k/\iplus^{k+1}$ as a free $\ocp$-module with a basis $\{P^n(\Pi)\}_{\absol{n}=k}$, then we have $\holkahler{k}\subset (L/\ol)^{\oplus \binom{d+k}{k}}\subset (\cp/\ocp)^{\oplus \binom{d+k}{k}}=\threekahler{(k)}{\okbar}{\ok}$ by the above corollary and the fact $\lambda_{n,x,P}(\pi)\in L$. This implies the injectivity of the above morphism. Let us prove (i). Since $\holkahler{k}$ is finitely generated over $\ol$, we have $\olkahler\subset (L/\ol)^{\oplus \binom{d+k}{k}}[p^n]$ for some $n$. By the structure theorem of finitely generated modules over discrete valuation rings, we obtain $\mu (\holkahler{k})\le\mu (p^{-n}\ol/\ol)^{\oplus\binom{d+k}{k}}=\binom{d+k}{k}$.
\end{proof}

\section{Good modules}\label{sec:good}
Throughout this section, let $L/K$ be an algebraic extension of local fields. In this section, we will investigate modules of special form over $\ol$, called ``good modules'' which play a crucial role in this paper.

\begin{definition}
An $\ol$-module $M$ is good if there exists a direct system $\{L_{\lambda}\}_{\lambda\in\Lambda}$ consisting of finite subextensions of $L/K$ and a direct system $\{M_{\lambda}\}_{\lambda\in\Lambda}$ consisting of $\mathcal{O}_{L_{\lambda}}$-submodules of $M$ satisfying the following conditions$:$
\begin{enuroman}
\item Transitive maps are inclusions, i.e., we have $L_{\lambda}\subset L_{\lambda '}$ and $M_{\lambda}\subset M_{\lambda '}$ for $\lambda <\lambda '$ and $L=\cup L_{\lambda},M=\cup M_{\lambda}$. Moreover, for $\lambda <\lambda '$, the canonical morphism $\mathcal{O}_{L_{\lambda '}}\otimes M_{\lambda}\to M_{\lambda '}$ is injective.

\item $M_{\lambda}$ is an $\mathcal{O}_{L_{\lambda}}$-module of finite length for $\lambda\in\Lambda$ and $\sup_{\lambda}{\mu(M_{\lambda})}<\infty$.
\end{enuroman}
\noindent We call $\{M_{\lambda}\}_{\lambda\in\Lambda}$ a direct system of $M$. For a direct system $\{M_{\lambda}\}_{\lambda\in\Lambda}$ of $M$, put $n(M)=\sup_{\lambda}{\mu (M_{\lambda})}$ and define $(e_i^{\lambda})_{1\le i\le n(M)}\in\mathbb{Q}_{\ge 0}^{n(M)}$ in order that they satisfy
\[
M_{\lambda}\cong\oplus_{1\le i\le n(M)}{\mathcal{O}_{L_{\lambda}}/p^{e_i^{\lambda}}\mathcal{O}_{L_{\lambda}}},\ e_1^{\lambda}\ge e_2^{\lambda}\ge\dotsb\ge e_{n(M)}^{\lambda},
\] 
where $p^{e_i^{\lambda}}$ denotes any element in $\mathcal{O}_{L_{\lambda}}$ with its $p$-adic valuation $e_i^{\lambda}$. Then let us define $r(M)$ by $r(M)=\#\{i\mid \sup_{\lambda}{e_i^{\lambda}=\infty}\}$.
\end{definition}

\begin{example}\label{ex:good}
\begin{enuroman}
\item $(L/\ol)^{\oplus n}$ is a good $\ol$-module with $r(M)=n(M)=n$.

\item For $k\in\mathbb{N}_{>0}$, $M=\holkahler{k}$ is a good $\ol$-module with $r(M)\le n(M)\le \binom{d+k}{k}$ by Theorem~\ref{fund kah}.

\item If $e_{L/K}=\infty$, then $L/\mathfrak{m}_L$ is not a good module. In fact, if it were a good module, we would have $L/\mathfrak{m}_L\cong L/\ol$ by Theorem~\ref{str} below. However, the annihilator by $\mathfrak{m}_{L}$ of both sides are $\ol/\mathfrak{m}_{L}$ and $0$, respectively. 
\end{enuroman}
\end{example}

\begin{remark}\label{rem:good}
\begin{enuroman}
\item Goodness is stable under sub, base change and direct sum: For a good $\ol$-module $M$, sub $\ol$-modules of $M$ are good. For an algebraic extension $L'/L$, $\mathcal{O}_{L'}\displaystyle\otimes_{\ol} M$ is a good $\mathcal{O}_{L'}$-module. The direct sum of two good $\ol$-modules is also good.

\item $n(M)$ and $r(M)$ are invariants of $M$, i.e., independent of the choice of direct systems. They are also compatible with base change. Indeed, it is easy in the case of $n(M)$. In the case of $r(M)$, it follows from Theorem~\ref{str} below.
\end{enuroman}
\end{remark}

\begin{lemma}
Let $n,r\in\mathbb{N}$ and $\phi :(\ol/p^n\ol)^{\oplus r}\to (\ol/p^{n+1}\ol)^{\oplus r}$ be an injective $\ol$-module homomorphism. Then there exists $\psi\in\mathrm{Aut}_{\ol}(\ol/p^{n+1}\ol)^{\oplus r}$ making the following diagram commutative$:$

\[\xymatrix{
(\ol/p^{n}\ol)^{\oplus r} \ar[r]^(.47){\phi} \ar[dr]_{p} & (\ol/p^{n+1}\ol)^{\oplus r} \ar[d]^{\psi} \\
 & (\ol/p^{n+1}\ol)^{\oplus r} \\
}\]
where $p$ is the $\ol$-module homomorphism characterized by
\[
p(0,\dotsc,\stackrel{i}{\check{1}},\dotsc,0)=(0,\dotsc,\stackrel{i}{\check{p}},\dotsc,0)\ ,\ 1\le i\le r.
\]
\end{lemma}
\begin{proof}
We can assume $L/K$ is finite. By the injectivity of $\phi$, we have a commutative diagram
\[\xymatrix{
(\ol/p^n\ol)^{\oplus r} \ar@{-->}[r]^(.45){\exists}  \ar[dr]_{\phi} & (p\ol/p^{n+1}\ol)^{\oplus r} \ar[d]^{\mathrm{inc}.} \\
 & (\ol/p^{n+1}\ol)^{\oplus r}, \\
}\]
where the dotted arrow is an isomorphism. Let $A\in \mathrm{M}_r(\ol)$ be a lift of $\phi$. Then, by the above diagram, we have $A=pB$ for some $B\in GL_r(\ol)$. Taking $\psi$ as the induced homomorphism by $B^{-1}$, we obtain the conclusion.
\end{proof}

\begin{corollary}\label{calculation good}
Let $\{M_n\}_{n\in\mathbb{N}}$ be a direct system of $\ol$-modules such that there exists $r\in\mathbb{N}$ with $M_n\cong (\ol/p^n\ol)^{\oplus r}$ for all $n$ and that the transitive maps $\{\phi_n:M_n\to M_{n+1}\}_{n\in\mathbb{N}}$ are all injective. Then
\[
\varinjlim_{n}{M_n}\cong (L/\ol)^{\oplus r}.
\] 
\end{corollary}
\begin{proof}
Choose an $\ol$-isomorphism $\psi_1:M_1\to (\ol/p\ol)^{\oplus r}$ and define inductively $\ol$-isomorphisms $\psi_n:M_n\to (\ol/p^n\ol)^{\oplus r}$, which make the following diagram commutative:
\[\xymatrix{
M_n  \ar[r]^{\phi_n} \ar[d]^{\psi_n} & M_{n+1} \ar[d]^{\psi_{n+1}} \\
(\ol/p^n\ol)^{\oplus r} \ar[r]^(.47){p} & (\ol/p^{n+1}\ol)^{\oplus r}.\\
}\]
Then
\[
\varinjlim{\psi_n}:\varinjlim{M_n}\stackrel{\cong}{\to}\varinjlim(\ol/p^n\ol)^{\oplus r}=(L/\ol)^{\oplus r}.
\]
\end{proof}

\begin{lemma}
Let $A$ be a discrete valuation ring with uniformizer $\pi$, let $M_1\subset M_2$ be $A$-modules of finite length and let $e_i^j,1\le i\le n,$ be non-negative integers satisfying
\[
M_j\cong\oplus_{1\le i\le n}{A/{\pi}^{e_i^j}A},\ e_1^j\ge e_2^j\ge\dotsb\ge e_n^j.
\]
Then $e_i^1\le e_i^2$ for $1\le i\le n$.
\end{lemma}
\begin{proof}
We have a commutative diagram
\[\xymatrix{
0 \ar[r] & M_1[\pi^{m-1}] \ar[r] \ar@{_{(}->}[d] & M_1[\pi^m] \ar[r] \ar@{_{(}->}[d]  & M_1[\pi^m]/M_1[\pi^{m-1}] \ar[r] \ar[d] & 0  \\
0 \ar[r] & M_2[\pi^{m-1}] \ar[r] & M_2[\pi^m] \ar[r] & M_2[\pi^m]/M_2[\pi^{m-1}] \ar[r] & 0  \\
}\]
with exact rows. Since $M_1[\pi^{m-1}]=M_1[\pi^{m}]\cap M_2[\pi^{m-1}]$, the right vertical map is also injective. Combining this with the equality $\dim_{k}{M_j[\pi^{m}]/M_j[\pi^{m-1}]}=\#\{i\mid e^j_i\ge m\}$, we obtain the assertion.
\end{proof}

\begin{corollary}\label{fundamental corollary}
Let $M$ be a good $\ol$-module with a direct system $\{M_{\lambda}\}$. Then we have $e^{\lambda}_i\le e^{\lambda '}_i$ for $\lambda <\lambda '$, $1\le i\le n(M)$.
\end{corollary}

\begin{theorem}[Structure theorem of good modules]\label{str}
Let $M$ be a good $\ol$-module.
\begin{enuroman}
\item For sufficiently large $m\in\mathbb{N}$, $M_{\mathrm{div}}=p^mM\cong (L/\ol)^{\oplus r(M)}$ and if $r(M)=n(M)$, then we can take $m=0$.

\item $M\cong (L/\ol)^{\oplus r}$ for some $r\in\mathbb{N}\Leftrightarrow r(M)=n(M)$.
\end{enuroman}
\end{theorem}

\begin{corollary}
Under the same hypothesis as above, we have
\[
T_pM\cong {\olhat}^{\oplus r(M)},\ V_pM\cong {\widehat{L}}^{\oplus r(M)},\ \dim_{\widehat{L}}{V_pM}=r(M).
\]
\end{corollary}

\begin{proof}[Proof of Theorem~\ref{str}]
We have only to prove (i). First, let us show $p^mM\cong (L/\ol)^{\oplus r(M)}$ for some $m\in\mathbb{N}$. For $m\in\mathbb{N}$, $p^mM$ is good and $r(M)=r(p^mM)$. Moreover, by taking $m$ sufficiently large (for example, $\sup_{r(M)<i}{\sup_{\lambda}{e^{\lambda}_i}}\le m$), we have $r(p^mM)=n(p^mM)$. So it is enough to prove $M\cong (L/\ol)^{\oplus r}$ under the condition $r=r(M)=n(M)$.

Choose $\{\lambda_N\}_{N\in\mathbb{N}}\subset\Lambda$ such that $\ {\lambda}_0 < {\lambda}_1 <\dotsb$ and $\ e^{{\lambda}_N}_r\ge N$ (Use Corollary~\ref{fundamental corollary}). For $N\in\mathbb{N}$, put $\Lambda_{N}=\{\lambda\in\Lambda\mid e^{\lambda}_r\ge N,\ \lambda>\lambda_N\}$, then $\Lambda_N$ is cofinal with $\Lambda$ by Corollary~\ref{fundamental corollary}. Since we have a canonical isomorphism $\mathcal{O}_{L_{\lambda}}\otimes_{\mathcal{O}_{L_{\lambda_{N}}}} (M_{\lambda_N}[p^N])\cong M_{\lambda}[p^N]$ for $\lambda\in\Lambda_N$, there are canonical isomorphisms
\[
M[p^N]\cong\displaystyle\varinjlim_{\lambda\in\Lambda_N}{M_\lambda[p^N]}\cong\displaystyle\varinjlim_{\lambda\in\Lambda_N}{\mathcal{O}_{L_{\lambda}}\otimes_{\mathcal{O}_{L_{\lambda_{N}}}} (M_{\lambda_N}[p^N])}\cong\ol\otimes_{\mathcal{O}_{L_{\lambda_{N}}}}{(M_{\lambda_N}[p^N])}.
\]
Since $M$ is $p$-torsion, we have
\[
M\cong\displaystyle\varinjlim_{N\in\mathbb{N}}{M[p^N]}\cong\displaystyle\varinjlim_{N\in\mathbb{N}}{\ol\otimes (M_{\lambda_N}[p^N])}\cong (L/\ol)^{\oplus r},
\]
where the last isomorphism follows from Corollary~\ref{calculation good}.
\end{proof}

In the rest of this section, we prove the exactness of the functor $V_p$ under certain assumption of goodness. Recall (\cite[p.$413$]{Sch2}) that, for a short exact sequence of abelian groups
\[\xymatrix{
0 \ar[r] & M_1 \ar[r] & M_2 \ar[r] & M_3 \ar[r] & 0\ , \\
}\]
we have a long exact sequence

\[\xymatrix{
0 \ar[r] & M_1[p^n] \ar[r] & M_2[p^n] \ar[r] & M_3[p^n] \ar[r] & M_1/p^nM_1 \ar[r] & M_2/p^nM_2 \ar[r] & M_3/p^nM_3 \ar[r] & 0 \\
}\]

\noindent which constitutes an inverse system. If $\{M_i[p^n]\}_{n\in\mathbb{N}}$ satisfies  the Mittag-Leffler condition $($ML$)$ for $i=1,2,3$, then we have an exact sequence

\[\xymatrix{
0\ar[r] & T_pM_1 \ar[r]&T_pM_2 \ar[r] & T_pM_3 \ar[r] & \widehat{M_1} \ar[r] & \widehat{M_2} \ar[r] & \widehat{M_3} \ar[r] & 0,
}\]

\noindent where $\widehat{\ }$ denotes the $p$-adic Hausdorff completion. For an abelian group $M$, ML is satisfied for $\{M[p^n]\}_{n\in\mathbb{N}}$ if $p^mM=0$ for some $m\in\mathbb{N}$ or $M$ is $p$-divisible. Therefore, if $M$ is a good $\ol$-module, $\{M[p^n]\}_{n\in\mathbb{N}}$ also satisfies ML.

\begin{lemma}\label{exactness vp}
Let
\[\xymatrix{
0 \ar[r] & M_1 \ar[r] & M_2 \ar[r] & M_3 \ar[r] & 0 \\
}\]
be an exact sequence of $\ol$-modules with $M_2$ good. Then
\[\xymatrix{
0 \ar[r] & V_pM_1 \ar[r] & V_pM_2 \ar[r] & V_pM_3 \ar[r] & 0 \\
}\]
is exact.
\end{lemma}
\begin{proof}
Since $V_p$ is left exact by the definition, we have only to prove the surjectivity. Since $V_p(p^nM)=V_pM$ for an abelian group $M$, we can assume $M_2$ is $p$-divisible by replacing $M_1$ by $M_1\cap p^nM_2$, $M_2$ and $M_3$ by $p^nM_2$, $p^nM_3$ for sufficiently large $n$. Since $M_3$ is $p$-divisible and $M_1$, $M_2$ are good, we have an exact sequence
\[\xymatrix{0 \ar[r]& T_pM_1 \ar[r] & T_pM_2 \ar[r] & T_pM_3 \ar[r] & \widehat{M_1}.
}\]
Since we have $\widehat{M_1}\cong M_1/p^nM_1$ for sufficiently large $n$, we obtain the conclusion by tensoring the above sequence with $\qp$.
\end{proof}

\begin{corollary}\label{four term lemma}
Let
\[\xymatrix{
0 \ar[r] & M_1 \ar[r] & M_2 \ar[r] & M_3 \ar[r] & M_4 \ar[r] & 0 \\
}\]
be an exact sequence of $\ol$-modules. If $M_2$ is good and $p^mM_4=0$ for some $m\in\mathbb{N}$, then
\[\xymatrix{
0 \ar[r] & V_pM_1 \ar[r] & V_pM_2 \ar[r] & V_pM_3 \ar[r] & 0  \\}\]
is exact.
\end{corollary}

\section{Main theorem}\label{sec:main}
\begin{definition}
For $k\in \mathbb{N}$, put $\mathcal{A}_{k}=\varprojlim_{n}{\hokbar{k-1}/p^n\hokbar{k}}$, i.e., the topological closure of $\hokbar{k-1}$ in $\BB{k}$ and let $d^{(k)}:\mathcal{A}_k\to \hkahler{k}$ be the canonical extension of $d^{(k)}:\hokbar{k-1} \to \hkahler{k}$ by continuity.
\end{definition}

We call a sequence $\{x_n\}\subset \BB{k}$ is $\BB{k}$-Cauchy if this sequence is a Cauchy sequence in $\BB{k}$.

\begin{lemma}
Let $k\in \mathbb{N}$.
\begin{enuroman}
\item $J_k\subset \mathcal{A}_k$.

\item There exists a commutative diagram:
\[\xymatrix{
0 \ar[r] & I_{+}^k/I_{+}^{k+1}  \ar[d]^{\mathrm{inc}.} \ar[r]^(.6){\mathrm{inc}.} & J_k  \ar[d]^{\mathrm{inc}.} \ar[r]^(.3){\mathrm{pr}.} & I^k/(I^{k+1}+I_{+}^{k}) \ar[d]^{\iota^{-1}}  \ar[r]  &  0    \\
0 \ar[r] & \AAA{k} \ar[r]^{\mathrm{inc}.} & \mathcal{A}_k \ar[r]^(.45){d^{(k)}} & \hkahler{k} \ar[r]  &  0   \\
}\]
with exact rows. $(\iota$ is the isomorphism in Theorem~\ref{thm Col}$(\mathrm{ii}).)$
\end{enuroman}
\end{lemma}
\begin{proof}
(i) Let $x\in J_k$. Then, by Theorem~\ref{thm Col}(iv), there exists a $B_k$-Cauchy sequence $\{x_n\}\subset\kbar$ which converges to $x$. Then, this sequence converges to zero when viewed as a sequence in $B_{k-1}$. Hence, by Theorem~\ref{thm Col}(iv), $x_n$ is contained in $\mathcal{O}_{\kbar}^{(k-1)}$ for sufficiently large $n$. Hence $x=\lim{x_n}$ is contained $\mathcal{A}_k$.

(ii) The exactness of the upper horizontal arrow is obvious. The lower horizontal arrow is obtained by taking the inverse limit of the exact sequence
\[\xymatrix{
0 \ar[r] & \mathcal{O}_{\kbar}^{(k)}/p^n\mathcal{O}_{\kbar}^{(k)} \ar[r] & \mathcal{O}_{\kbar}^{(k-1)}/p^n\mathcal{O}_{\kbar}^{(k)} \ar[r] & \threekahler{(k)}{\okbar}{\ok}  \ar[r] & 0, \\
}\]
hence it is also exact. The commutativity of the left square is obvious and it suffices to prove the commutativity of the right square. For $x\in J_k$, write $x=\lim{x_n}$ with $x_n\in \hokbar{k-1}=\kbar\cap (\A+I^k)$. Write $x_n=y_n+z_n$, $y_n\in \A$, $z_n\in I^k$, then we have $\iota\circ d^{(k)}(x)=z_n \mod{I^{k+1}+I_{+}^k}$ for $n\gg 0$. Since $x-\overline{x_n}\in \AAA{k}$ for $n\gg 0$, we have $x-\overline{z_n}=x-\overline{x_n}+\overline{y_n}\in \AAA{k}\cap J_k=I_{+}^k/I_{+}^{k+1}$ where $\Bar{\ }\Bar{\ }$ denotes $\mod{I^{k+1}}$.
\end{proof}

Put again $d^{(k)}:J_k\to \hkahler{k}$ the restriction of $d^{(k)}:\mathcal{A}_k\to \hkahler{k}$ to $J_k$.

\begin{corollary}
$d^{(k)}:J_k\to \hkahler{k}$ is a surjective $\ocp$-module homomorphism.
\end{corollary}

\begin{definition}
For $k\in \mathbb{N}$, put $\lhat{k}=\varprojlim_{n}{L/p^n\hol{k}}$ $(resp.\ \widehat{L}_{\infty}=\varprojlim_{n,k}{L/p^n\hol{k}})$, i.e., the topological closure of $L$ in $\BB{k}$ $(resp.\ \bdr)$. Note that $\lhat{0}$ is just $\widehat{L}$.
\end{definition}

Note that for $x\in J_k\cap \lhat{k}$, we have a $\BB{k}$-Cauchy sequence $\{x_n\}$ such that $x=\lim{x_n}$ with $x\in \hol{k-1}$ as in the proof of previous lemma. Let us note the following lemma:

\begin{lemma}
Let $F$ be a non-archimedean complete valuation field and let $F_0$ be a dense subfield of $F$. Let $V$ be a complete topological $F$-vector space and let $V_0$ be a sub $F_0$-vector space of $V$ which is closed in $V$. Then $V_0$ is also a sub $F$-vector space of $V$.
\end{lemma}

In particular, $J_k\cap\lhat{k}$ is an $\widehat{L}$-vector space and we have a well-defined $\mathcal{O}_{\widehat{L}}$-module homomorphism $d^{(k)}:J_k\cap \lhat{k}\to (\olkahler)_{\mathrm{div}}$. The image is contained in $d^{(k)}(\hol{k-1})$.

\begin{definition}
For $k\in \mathbb{N}_{>0},\ L/K$ is said to be de Rham at level $k$ if $d^{(k)}(\hol{k-1})$ contains $\divi{\olkahler}$ and put $\drcoh{k}=\divi{\olkahler}/d^{(k)}(J_k\cap \lhat{k})$.
\end{definition}

\begin{remark}\label{rem:dr}
\begin{enuroman}
\item The definition of de Rham at level $k$ in this paper corresponds to that of de Rham at level $k+1$ in \cite{IZ}. Our numbering is natural as we will see in Theorem~\ref{thm:main}.     

\item $L/K$ is  de Rham at level $k$ if and only if $\drcoh{k}=0$. In fact, assume that $L/K$ is de Rham at level $k$ and let $\omega\in\divi{\threekahler{(k)}{\ol}{\ok}}$. Then we can take $\omega_n\in\divi{\threekahler{(k)}{\ol}{\ok}}$ and $x_n\in\mathcal{O}_{L}^{(k-1)}$ such that $\omega_0=\omega,\ \omega_n=p\omega_{n+1}$ and $\omega_n=d^{(k)}(x_n)$. Then, since $d^{(k)}(px_{n+1}-x_n)=p\omega_{n+1}-\omega_n=0,\ px_{n+1}-x_n$ is contained in $\mathcal{O}_{L}^{(k)}$. Hence we have $p^{n+1}x_{n+1}-p^nx_n\in p^n\mathcal{O}_{L}^{(k)}$ and so $\{p^nx_n\}$ is a $B_k$-Cauchy sequence. On the other hand, since $p^nx_n\in p^n\mathcal{O}_{L}^{(k-1)}$, $\{p^nx_n\}$ converges to zero in $B_{k-1}$. So we have $d^{(k)}(x)=\omega$ with $x=\lim{p^nx_n}\in J_k\cap\widehat{L}_{k}$. 

By this argument and the $\ol$-linearity of $d^{(k)}:J_k\cap\widehat{L}_k\to\divi{\threekahler{k}{\ol}{\ok}}$, we can also prove $\drcoh{k}=0$ if $\mathrm{Im}(d^{(k)}:\mathcal{O}_{L}^{(k-1)}\to \threekahler{(k)}{\ol}{\ok})_{\mathrm{div}}$ generates $\divi{\threekahler{(k)}{\ol}{\ok}}$ over $\ol$.
\end{enuroman}
\end{remark}

\begin{lemma}\label{valuationw}
Let $L/K$ be de Rham at level $k$. Then there exists a constant $m_k\in \mathbb{N}$, which depends only on $L$, satisfying the following$:$

 For $x\in \mathrm{Im}(\lhat{k}\to \lhat{k-1})$, there exists a lift $x'\in \lhat{k}$ of $x$ such that $w_k(x')\ge w_{k-1}(x)-m_k$. 
\end{lemma}
\begin{proof}
Take $m_k$ satisfying $\pi_K^{m_k}\olkahler=\divi{\olkahler}$. Since $w_k(\pi_K^na)=w_k(a)+n$ for all $k$, we can assume $w_{k-1}(x)=0$. Let $x'\in \lhat{k}$ be a lift of $x$. Choose $\{x_n\}\subset J_k\cap \lhat{k}$ such that $\pi_K^{m_k}d^{(k)}(x')=\pi_K^{m_k}d^{(k)}(x_0)$, $d^{(k)}(x_n)=\pi_Kd^{(k)}(x_{n+1})$. Then $\{\pi_K^nx_n\}$ is $\BB{k}$-Cauchy and we have $x''=\lim{\pi_K^nx_n}\in J_k\cap\lhat{k}$. We have $\pi_K^{m_k}d^{(k)}(x')=\pi_K^{m_k}d^{(k)}(x'')$, so the modified lift $x'-x''$ satisfies the required condition. 
\end{proof}

In particular, a $\BB{k-1}$-Cauchy sequence of $L$ lifts to a $\BB{k}$-Cauchy sequence of $\lhat{k}$. Hence we have

\begin{corollary}\label{de Rham is surj}
Under the same assumption as above, the canonical map $\lhat{k}\to \lhat{k-1}$ is surjective.
\end{corollary}

\begin{theorem}[Main theorem]\label{thm:main}
For $k\in\mathbb{N}_{>0}$, the followings are equivalent$:$
\begin{enuroman}
\item For $1\le n\le k$, $\lhat{n}=\BB{n}^{\gl}$.  

\item For $1\le n\le k$, $\jl{n}=J_n^{\gl}$.

\item For $1\le n\le k$, $V_p(\holkahler{n})=J_n^{\gl}$ and $L/K$ is de Rham at level $n$.
\end{enuroman}
\end{theorem}

Before the proof, we prepare an easy lemma:

\begin{lemma}\label{easy}
Let $F$ be a non-trivial non-archimedean complete valuation field. Let $m,m'\in\mathbb{N}$ and let $\overline{\phi}:F^{\oplus m}\to (F/\of)^{\oplus m'}$ be an $\of$-module homomorphism. Then there exists a $F$-vector space homomorphism $\phi:F^{\oplus m}\to F^{\oplus m'}$ which makes the following diagram commutative$:$
\[\xymatrix{
  &  F^{\oplus m}  \ar@{-->}[dl]_(.47){\phi} \ar[d]^(.47){\overline{\phi}}  \\
F^{\oplus m'}  \ar[r]_(.4){\mathrm{pr}.} & (F/\of)^{\oplus m'}  \\
}\]
\end{lemma}
\begin{proof}
We can assume that $m=m'=1$ and $\overline{\phi}(\of)=0$. Take a non-zero element $\pi$ in the maximal ideal of $\mathcal{O}_F$ and choose $x_i\in \of$ such that $\overline{\phi}(1/{\pi}^i)=(1/{\pi}^i) x_i \mod{\of}$, then $\{x_i\}$ is Cauchy, hence $\phi(1)=\lim{x_i}$ satisfies the condition.  
\end{proof}

\begin{proof}[Proof of Theorem~\ref{thm:main}]
(i)$\Rightarrow$(ii) This follows from a commutative diagram with exact rows
\begin{equation}\label{diag}
\xymatrix{
0 \ar[r] & J_n\cap\lhat{n}  \ar[d]^{\mathrm{inc}.} \ar[r]^(.55){\mathrm{inc}.} & \lhat{n}  \ar[d]^{\mathrm{inc}.} \ar[r]^{\mathrm{pr}.} & \lhat{n-1}  \ar[d]^{\mathrm{inc}.}  \\
0 \ar[r] & J_n^{\gl} \ar[r]^{\mathrm{inc}.} & \BB{n}^{\gl} \ar[r]^{\mathrm{pr}.} & \BB{n-1}^{\gl}. \\
}
\end{equation}

(ii)$\Rightarrow$(iii) Put $m=\dim_{\widehat{L}}{J_n\cap\lhat{n}}=\dim_{\widehat{L}}{J_n^{\gl}}$, $m'=\dim_{\widehat{L}}{V_p(\holkahler{n})}$. Then $m\ge m'$ by Remark~\ref{rem:kah}. By the structure theorem of good modules, we can identify $\divi{\holkahler{n}}=(\widehat{L}/\olhat)^{\oplus m'}$. Since $J_n\cap\lhat{n}\cong\widehat{L}^{\oplus m}$, we have a commutative diagram by the previous lemma
\[\xymatrix{
 & & & \widehat{L}^{\oplus m'} \ar[d]^{\mathrm{pr}.} \\
0 \ar[r] & \ker{d^{(n)}} \ar[r]^{\mathrm{inc}.} & \jl{n} \ar[r]^{d^{(n)}} \ar@{-->}[ur]^{\exists D^{(n)}} & \divi{\holkahler{n}}  \\
 & \ker{D^{(n)}} \ar@{^{(}->}[u] \ar@{^{(}->}[ur] &  \\
}\]
with an exact row. Since $\ker{D^{(n)}}\subset \divi{\ker{d^{(n)}}}\subset\divi{I_{+}^{n}/I_{+}^{n+1}}=0$, we have $m'\ge m$, i.e., $m=m'$.

(iii)$\Rightarrow$(ii) Put $m=\dim_{\widehat{L}}{J_n\cap\lhat{n}}$, $m'=\dim_{\widehat{L}}{V_p(\holkahler{n})}$. By the surjective $\olhat$-module homomorphism 
\[
{\widehat{L}}^{\oplus{m}}\cong J_n\cap\lhat{n}\to (\holkahler{n})_{\mathrm{div}}\cong (\widehat{L}/\olhat)^{\oplus m'},
\]
we have $m\ge m'$. Since $J_n\cap\lhat{n}\subset J_n^{\gl}$, we have the conclusion.

(ii)+(iii)$\Rightarrow$(i) This follows from the diagram (\ref{diag}), Theorem~\ref{ax} and Corollary~\ref{de Rham is surj}.
\end{proof}

\section{Deeply ramified extensions and shallowly ramified extensions}\label{sec:deep}
For a finite algebraic extension $L/K$, denote $\dif{L}{K}$ the different ideal of $L/K$. This is the fitting ideal of $\threekahler{1}{\ol}{\ok}$ (\cite[Lemma~1.1]{Fal}).
\begin{definition}\label{def:deep}
An algebraic extension $L/K$ is deeply ramified $($resp. shallowly ramified $)$ if $\dim_{\widehat{L}}{V_p(\lkahler)}$

\noindent $=d+1$ $($resp. $0)$.
\end{definition}
Note that, for a finite subextension $K'/K$ of $L/K$ (resp. a finite extension $L'/L$), $L/K$ is deeply ramified or shallowly ramified if and only if $L/K'$ (resp. $L'/K$) is so. In fact, we have $r(\threekahler{1}{\ol}{\ok})=r(\threekahler{1}{\ol}{\mathcal{O}_{K'}})$ (resp. $r(\threekahler{1}{\ol}{\ok})=r(\threekahler{1}{\mathcal{O}_{L'}}{\mathcal{O}_{K}})$): In the case of $\threekahler{1}{\ol}{\mathcal{O}_{K'}}$, this follows from Lemma~\ref{exactness kahler} and Lemma~\ref{exactness vp}. Let us consider the case of $r(\threekahler{1}{\mathcal{O}_{L'}}{\mathcal{O}_{K}})$. By \cite[V, $\S 4$, Lemma~6]{Se1}, we may assume that there exist finite extensions $K''/K'/K$ such that $K''$ and $L$ are linearly disjoint over $K'$ and that $K''L=L'$. Replacing $K$ by $K'$, we may assume there exists a finite extension $K'/K$ such that $L'=K'L$ and that $L,K'$ are linearly disjoint over $K$. Let $\{L_{\lambda}\}$ be the finite subextensions of $L/K$ and put $L'_\lambda=L'L_{\lambda}$. Then, by the definition of different via the trace form (\cite[III, $\S 3$]{Se1}), we have $v_p(\dif{K'}{K})\ge v_p(\dif{L^{'}_{\lambda}}{L_{\lambda}})$. Hence $\threekahler{1}{\mathcal{O}_{L'}}{\ol}=\varinjlim_{\lambda}{\threekahler{1}{\mathcal{O}_{L^{'}_{\lambda}}}{\mathcal{O}_{L_{\lambda}}}}$ is killed by some power of $p$ and this implies the equality $r(\threekahler{1}{\ol}{\mathcal{O}_{K}})=r(\threekahler{1}{\mathcal{O}_{L'}}{\mathcal{O}_{K}})$ by Lemma~\ref{exactness vp}. Also, it is easy to see that, if $L/K$ is  deeply ramified, then $L'/K$ is deeply ramified for all algebraic extensions $L'/L$.

\newcommand{\coz}{\mathcal{O}_{K}[\zeta_{p^n}]}
\newcommand{\cott}[1]{\mathcal{O}_{K}[t_{#1}^{p^{-n}}]}
\newcommand{\oi}{\mathcal{O}_{K}}

\begin{example}\label{ex:deep}
Let $t_1,\dotsc,t_d$ be a $p$-basis of $K_0$. Put 
\[
K_{n}=K(\zeta_{p^n},t_1^{p^{-n}},\dotsc,t_d^{p^{-n}}),\ K_{\infty}=\cup{K_n},
\]
where $\zeta_{p^n}$ is a primitive $p^n$-th root of unity. We prove that $K_{\infty}$ is deeply ramified over $K$. We can assume $K=K_0$. Note that 
\[
\mathcal{O}_{K_n}\cong\ok [\zeta_{p^n}]\otimes_{\ok}\cott{1}\otimes_{\ok}\dotsb\otimes_{\ok}\cott{d}.
\]
By a simple calculation, we have 
\begin{align*}
\threekahler{1}{\coz}{\oi} & \cong(\coz/p^{n-\frac{1}{p-1}})\mathrm{dlog}\zeta_{p^n},\\
\threekahler{1}{\cott{j}}{\oi} & \cong(\cott{j}/p^n)\mathrm{dlog}t_j^{p^{-n}},\ 1\le j\le d.
\end{align*}
where $\mathrm{dlog}\zeta_{p^n}=(1/\zeta_{p^n}) d\zeta_{p^n}$ and $\mathrm{dlog}t_{j}^{p^{-n}}=(1/t_{j}^{p^{-n}}) dt_{j}^{p^{-n}}$. 

Thus, we have 
\[
\threekahler{1}{\mathcal{O}_{K_n}}{\oi}\cong(\mathcal{O}_{K_n}/p^{n-\frac{1}{p-1}})\mathrm{dlog}\zeta_{p^n}\oplus\oplus_{1\le j\le d}{(\mathcal{O}_{K_n}/p^n)\mathrm{dlog}t_j^{p^{-n}}},
\]
which implies that $K_{\infty}/K_0$ is deeply ramified. 

As an obvious example, $\kbar/K$ is deeply ramified. Typical examples of shallowly ramified extension are unramified extensions and tamely ramified extensions.
\end{example}

\begin{theorem}\label{thm:deep}
For an algebraic extension $L/K$, the followings are equivalent;
\begin{enumerate}
\item[\rm{(i)}] $L/K$ is deeply ramified. 

\item[\rm{(ii)}] For all finite extensions $L'/L$, $\tr{L'}{L}(\mathfrak{m}_{L'})=\mathfrak{m}_{L}$.

\item[\rm{(iii)}] For all algebraic extensions $L'/L$, $\threekahler{1}{\mathcal{O}_{L'}}{\ol}=0$.

\item[\rm{(iii)'}] For all algebraic extensions $L'/L$, $\threekahler{1}{\mathcal{O}_{L'}}{\ol}$ is almost zero.

\item[\rm{(iii)''}] There exists an integer $m$ such that, for all algebraic extensions $L'/L$, $p^m\threekahler{1}{\mathcal{O}_{L'}}{\ol}=~0$.

\item[\rm{(iv)}] $\lkahler\cong (L/\ol)^{\oplus d+1}$.
\end{enumerate}
\end{theorem}

\begin{proof}
\newcommand{\oo}[2]{\mathcal{O}_{{#1}_{#2}}}
\newcommand{\ooo}[1]{\mathcal{O}_{#1}}
First, we prove the equivalence except (ii).

(i) and (iv) is equivalent by Theorem~\ref{str}. Obviously (iii)$\Rightarrow$(iii)'$\Rightarrow$(iii)''. Using Lemma~\ref{exactness kahler} to $\kbar/L/K$ and applying Lemma~\ref{exactness vp}, (iii)'' implies $r(\threekahler{1}{\ol}{\ok})=r(\threekahler{1}{\okbar}{\ok})=d$, i.e., (iv). To prove (iv)$\Rightarrow$(iii), by Lemma~\ref{exactness kahler}, we have only to prove that, for any algebraic extension $L'/K$, an injective $\mathcal{O}_{L'}$-module homomorphism $\iota :(L'/\mathcal{O}_{L'})^{\oplus n}\to (L'/\mathcal{O}_{L'})^{\oplus n}$ is surjective: It suffices to prove the surjectivity after taking the $p^m$-torsion part of both sides and we can assume $L'/K$ is finite. Then the length of the cokernel of $\iota [p^m]$ is $0$, i.e., $\iota [p^m]$ is surjective.

We will finish the proof by proving (i)$\Rightarrow$(ii) and (ii)$\Rightarrow$(iii)'.

(i)$\Rightarrow$(ii) Obviously, we can choose a tower $\{L_n\}$ of finite subextensions of $L/K$ such that, for each $n$, there exists a surjective $\oo{L}{n+1}$-module homomorphism $\threekahler{1}{\oo{L}{n+1}}{\oo{L}{n}}\to (\oo{L}{n+1}/p\oo{L}{n+1})^{\oplus d+1}.$ By replacing $K$ by a finite subextension of $K$ in $L$, we can take a finite extension $K'/K$ such that linearly disjoint from $L$ over $K$ and $L'=K'L$. Then by \cite[Theorem~$1.2$]{Fal} and \cite[Proposition~$9$]{Tate}, we have $\tr{L^{'}_{\infty}}{L_{\infty}}(\mathfrak{m}_{L^{'}_{\infty}})=\mathfrak{m}_{L_{\infty}}$, where $L_{\infty}=\cup{L_n}$ and ${L'}_{\infty}=\cup{L_nK'}$. For any finite subextension $L_{\lambda}/K$ of $L/K$, we can take the above tower $\{L_n\}$ satisfing $L_{\lambda}\subset L_{\infty}$. Then we have $\mathfrak{m}_{L_{\lambda}}\subset\mathfrak{m}_{L_{\infty}}=\mathrm{Tr}_{L^{'}_{\infty}/L_{\infty}}(\mathfrak{m}_{L^{'}_{\infty}})\subset\mathrm{Tr}_{L'/L}(\mathfrak{m}_{L'})$. Hence we have $\tr{L'}{L}(\mathfrak{m}_{L'})=\mathfrak{m}_{L}$.

(ii)$\Rightarrow$ (iii)' First, note that (ii) implies $e_{L/K}=\infty$: If not, by replacing $K$ by a finite subextension $K_1$ of $K$ in $L$ with $e_{L/K}=e_{K_1/K}$, we may assume $e_{L/K}=1$. Put $L'=L(\pi_K^{p^{-1}})$ and $K'=K(\pi_K^{p^{-1}})$. Then, for all finite subextensions $L_{\lambda}/K$ of $L/K$, $\mathcal{O}_{L_{\lambda}K'}=\mathcal{O}_{L_{\lambda}}\otimes \mathcal{O}_{K'}$ since $e_{L_{\lambda}/K}=1$ and $e_{K'/K}=[K':K]=p$. Denote the integer part of $x$ by $[x]$. Then we have (\cite[III, $\S 3$, Proposition~7]{Se1})
\begin{align*}
v_{L_{\lambda}}(\tr{L_{\lambda}K'}{L_{\lambda}}(\mathfrak{m}_{L_{\lambda}K'})) & =[v_{L_{\lambda}}(\mathfrak{m}_{L_{\lambda}K'})+v_{L_{\lambda}}(\mathfrak{D}_{L_{\lambda}K'/L_{\lambda}})] \\
 & =\left[\frac{1}{p}+e_K+\frac{p-1}{p}\right] =e_K+1>1,
\end{align*}
which is a contradiction.

Next, we prove that (ii) implies the following claim:

\begin{claim}\label{claim}
For any finite extensions $K'/K$ which is linearly disjoint from $L$ over $K$, there exists a tower $\{L_n\}$ of finite subextensions of $L/K$ such that $v_p(\dif{L_nK'}{L_n})\to 0$ $(n\to\infty)$.
\end{claim}

Note that, for any tower $\{L_n\}$ of finite subextensions of $L/K$, $\{v_p(\dif{L_nK'}{L_n})\}_{n}$ is a decreasing sequence. Choose a tower $\{L_n\}$ such that $\{\tr{L_nK'}{L_n}(\mathfrak{m}_{L_nK'})\}_{n}$ generate $\mathfrak{m}_{L}$ over $\ol$: Such one exists since $\mathfrak{m}_{L}$ is countably generated over $\ol$. We will prove that this $\{L_n\}$ satisfies the condition.

If $v_p(\dif{L_nK'}{L_n})\to\varepsilon\ (n\to\infty)$ with $\varepsilon\neq 0$, we can choose sufficiently large $n_0$ such that $1/e_{L_{n_0}}\le v_p(\dif{L_nK'}{L_n})$ for all $n\gg 0$. Then we have
\begin{align*}
v_K(\tr{L_nK'}{L_n}(\mathfrak{m}_{L_nK'})) & =[v_{L_n}(\mathfrak{m}_{L_nK'})+v_{L_n}(\mathfrak{D}_{L_nK'/L_n})]\frac{1}{e_{L_n/K}} \\
 & \ge \left[v_{L_n}(\mathfrak{m}_{L_nK'})+e_{L_n}\frac{1}{e_{L_{n_{0}}}}\right]\frac{1}{e_{L_n/K}} \ge\frac{1}{e_{L_{n_0}/K}},
\end{align*}
which contradicts the assumption that the left hand side converges to $0$ as $n\to\infty$. Hence we have proved the claim. 

Finally, let us prove (iii)'. To prove this, we may replace $K$ by a finite extension of $K$ in $L$. So we may assume that there exists a finite subextension $K'/K$ in $L'/K$ such that $K'$ and $L$ are linearly disjoint over $K$ and that $L'=LK'$. Let $L_{\lambda}/K$ be any finite subextension of $L/K$ and put $L^{'}_{\lambda}=L_{\lambda}K'$. Then, by applying the claim to the finite extension $L^{'}_{\lambda}/L_{\lambda}$, we can choose a tower $\{L_n\}$ satisfying the conclusion of the claim with $L_{\lambda}\subset L_{\infty}=\cup{L_n}$. Since the canonical map $\threekahler{1}{\ooo{L^{'}_{\lambda}}}{\ooo{L_{\lambda}}}\to\threekahler{1}{\ooo{L'}}{\ooo{L}}$ factors through the almost zero module $\threekahler{1}{\mathcal{O}_{L^{'}_{\infty}}}{\oo{L}{\infty}}\!\!=\varinjlim_{n}{\threekahler{1}{\mathcal{O}_{L_nK'}}{\mathcal{O}_{L_n}}}$, the assertion is proved.
\end{proof}

\begin{proposition}\label{prop:deep}
Let $L/K$ be deeply ramified. Then
\begin{enuroman}
\item $\lkahler$ is $p$-divisible.

\item $v_{p}(e_{L/K})=\infty$, i.e., $\sup_{L'}{v_p(e_{L'/K})}=\infty$ where $L'$ runs through all finite subextensions of $L/K$.

\item $k_{L}$ is perfect.

\item $\coh{1}{\gl}{\mkbar}=0$, where $\mathfrak{m}_{\kbar}$ is endowed with the discrete topology.

\item Put $\Delta_{L}(x)=\inf_{\sigma\in\gl}{v_p(x^{\sigma}-x)}$ for $x\in\kbar$ $($cf. \cite{Ax}$)$. Then $\Delta_L(x)=\sup_{a\in L}{v_p(x-a)}$ for all $x\in\kbar$.

\item All $\cp$-representations of $\gl$ are admissible and $\coh{k}{\gl}{V}=0$, $k>0$, for all $\bdr$-representations $V$ of $\gl$.
\end{enuroman}
\end{proposition}

We only use the property (vi) in the following.

\begin{proof}
(i) This follows from Theorem~\ref{thm:deep}(iv).

(ii) Assume $v_p(e_{L/K})<\infty$ and deduce a contradiction. By replacing $K$ by its finite extension $K'$ if necessary, we can assume $v_p(e_{L/K})=0$. Choose a finite subextension $L'/K$ of $L/K$ such that the $p$-adic valuations $\{e_1,\dotsc,e_{d+1}\}$ of the invariant factors of $\llkahler$ satisfy $e_1\ge\dotsb\ge e_{d+1}>1$. By replacing $K$ by the maximal unramified extension of $L'/K$, we can assume that $L'/K$ satisfies Hypothesis~\ref{hyp}. Now use the same notation as $\S\ref{fund property}$. Since $P_0(X)\equiv X_0^{e_{L/K}}\mod{\pi_K\ok[X_0,\dotsc,X_d]}$ by the construction of $P_0(X)$ and $dP_0(\pi)=0$, we have
\[
(e_{L'/K}\pi_{L'}^{e_{L'/K}-1}+a\pi_K)d\pi_{L'}\in\sum_{0<i\le d}{\ol d\pi_i}
\]
for some $a\in\mathcal{O}_{L'}$. Hence we have $\mu(p\;\!\llkahler)\le d$, this contradicts to the assumption of $e_i$'s.

(iii) Assume that $k_L$ is not perfect. Choose $t\in\ol$ with $\bar{t}\in k_L\setminus k_L^p$ and put $L'=L(t^{p^{-1}})$. Then one can prove $\mathcal{O}_{L'}=\ol [t^{p^{-1}}]$, in particular $\threekahler{1}{\mathcal{O}_{L'}}{\ol}\neq 0$, which contradicts to Theorem~$\ref{thm:deep}$(iii).

(iv) By \cite[Theorem~2.4]{Fal}, $\coh{1}{\gl}{\okbar}$ is almost zero. If $\coh{1}{\gl}{\mathfrak{m}_{\kbar}}\neq 0$, we have $x\in\kbar$ such that the $1$-cocycle $s$ defined by $x$ is not a $1$-coboundary. Choose $\varepsilon\in v_p(L)$ such that $\inf_{\sigma\in\gl}{v_p(x^{\sigma}-x)}>2\varepsilon>0$. If we consider the $1$-cocycle defined by $x/p^{2\varepsilon}$, this cocycle class in $\coh{1}{\gl}{\okbar}$ is killed by $p^{\varepsilon}$, so we have $x-p^{\varepsilon}x'\in L$ for some $x'\in\okbar$. Hence $s$ is a $1$-coboundary, which contradicts to the definition of $s$. 

(v) Assume that there exists an element $x\in\kbar$ with $\Delta_L(x)>\sup_{a\in L}{v_p(x-a)}$. By multiplying with some element in $L\setminus\{0\}$ if necessary, we may assume the inequalities $\inf\nolimits_{\sigma\in\gl}{v_p(x^{\sigma}-x)}>0>\sup\nolimits_{a\in L}{v_p(x-a)}$. Then, since the $1$-cocycle defined by $x$ is killed in $\coh{1}{\gl}{\mkbar}$ by (iv), there exists $a\in L$ such that $x-a\in\mathfrak{m}_{\kbar}$. this contradicts to the above inequality.

(vi) We can reduces to the case that $V$ is a $\cp$-representation. Since all $\cp$-representations of $\gl$ are admissible by Theorem~\ref{thm:deep}(ii) and the argument of \cite[Proposition~4]{Sen2}, it suffices to prove the equality $\coh{k}{\gl}{\cp}=0$ for $k>0$. We can prove this by the argument in \cite[(3.2)]{Tate}, using Theorem~\ref{thm:deep}(ii).
\end{proof}

\begin{remark}
Let $\{G_K^{a}\}_{a\in\mathbb{Q}_{\ge 0}}$ be the ramification filtration of Abbes-Saito (\cite{AS}). An algebraic extension $L/K$ has a finite conductor if $L\subset \kbar^{G_K^{a}}$ for some $a$. If $L$ has a finite conductor, then $L/K$ is shallowly ramified by \cite[Proposition~7.3]{AS}. In classical case, i.e., $k_K$ is perfect, the converse is also true (\cite[p.143]{CG}). In imperfect residue field case, the author does not know this is true. For other properties of deeply ramified extensions and geometric applications in classical case, see \cite{CG}. 
\end{remark}

\begin{theorem}\label{thm:more}
Let $L/K$ be a deeply ramified extension. Put $I_k=I/I^{k+1}\subset\BB{k}$.
\begin{enuRoman}
\item In addition to the conditions of Theorem~\ref{thm:main}, the following are equivalent:
\begin{enumerate}
\item[\rm{(i)'}] $\lhat{k}=B_k^{\gl}$.

\item[\rm{(ii)'}]  For $1\le n\le k$, $J_n\cap\lhat{n}$ generates $J_n$.

\item[\rm{(iii)'}] For $1\le n\le k$, $V_p(\threekahler{(n)}{\ol}{\ok})=J_n^{\gl}$ and $d^{(n)}:\mathcal{O}_{L}^{(n-1)}\to\threekahler{(n)}{\ol}{\ok}$ is surjective.

\item[\rm{(iv)}] For $1\le n\le k$, $I_n\cap\lhat{n}=I_n^{\gl}$.

\item[\rm{(iv)'}] For $1\le n\le k$, $I_n\cap\lhat{n}$ generates $I_n$.

\item[\rm{(v)}] $I_k\cap\lhat{k}=I_k^{\gl}$.

\item[\rm{(v)'}] $I_k\cap\lhat{k}$ generates $I_k$.
\end{enumerate}

\item The followings are equivalent:
\begin{enumerate}
\item[\rm{(i)}] $\lhat{\infty}=(\bdr)^{\gl}$.

\item[\rm{(ii)}] For all $k\in\mathbb{N}$, $\lhat{k}=\BB{k}^{\gl}$.

\item[\rm{(iii)}] $I\cap\lhat{\infty}=I^{\gl}$.

\item[\rm{(iii)'}] $I\cap\lhat{\infty}$ generates $I$.
\end{enumerate}
\end{enuRoman}
\end{theorem}
\begin{proof}
(I) (i)$\Rightarrow$(i)' is obvious and (i)'$\Rightarrow$(i) follows from the surjectivity of $B_n^{\gl}\to B_{n-1}^{\gl}$.

(ii)$\Leftrightarrow$(ii)' A direct consequence of Proposition~\ref{prop:deep}(vi). 

(iii)$\Rightarrow$(iii)' follows from Theorem~\ref{fund kah}, Theorem~\ref{str} and Proposition~\ref{prop:deep}(vi) and (iii)'$\Rightarrow$(iii) is obvious. 

(ii)$\Rightarrow$(iv) We will prove 
\[
I_n^i\cap\lhat{n}=(I_n^i)^{\gl}\ \mathrm{for}\ 1\le i\le n
\]
by induction on $n$: In the case $n=1$, there is nothing to prove. For general $n$, we use descending induction on $i$, starting from $i=n$: In the case $i=n$, there is nothing to prove. For general $i$, the conclusion follows from the following commutative diagram with exact rows
\[\xymatrix{
0 \ar[r] & I_n^i\cap\lhat{n} \ar[r] \ar[d] & I_n^{i-1}\cap\lhat{n} \ar[r] \ar[d] & J_{i-1}\cap\lhat{i-1} \ar[d] \ar[r] & 0 \\
0 \ar[r] & (I_n^i)^{\gl} \ar[r] & (I_n^{i-1})^{\gl} \ar[r] & J_{i-1}^{\gl} \ar[r] & 0, \\
}\]
where the surjection in the upper row follows from Corollary~\ref{de Rham is surj}.

(iv)$\Rightarrow$(iv)', (v)$\Rightarrow$(v)' Since the projection $I_n^{\gl}\to I_1^{\gl}$ is surjective and $I_1\cong\cp\otimes (I_1^{\gl})$, we have the conclusion by Nakayama's lemma.

(iv)$\Rightarrow$(v), (iv)'$\Rightarrow$(v)' Obvious.

(v)'$\Rightarrow$(ii)' From $(I_n\cap\lhat{n})^n\subset{J_n}\cap\lhat{n}$ and  $I_n^n=J_n$.

\noindent (II) (i)$\Rightarrow$(iii)  Obvious.

(iii)$\Rightarrow$(iii)' We have the canonical surjection $I^{\gl}=\varprojlim{(I_k^{\gl})}\to I_1^{\gl}$ and the canonical isomorphism $I_1\cong\cp\otimes (I_1^{\gl})$. From these, we see that $I^{\gl}$ generates $I$ by Nakayama's lemma.

(iii)'$\Rightarrow$(ii) Since the canonical projection $I\to I_k$ is surjective, the condition (v)' of (I) holds for all $k\in\mathbb{N}$.

(ii)$\Rightarrow$(i) \xymatrix{(\bdr)^{\gl}=\varprojlim_{k}{\BB{k}^{\gl}}=\varprojlim_{k}{\lhat{k}}=\varprojlim_{k}{\varprojlim_{n}{L/p^n\hol{k}}}=\varprojlim_{k,n}{L/p^n\hol{k}}=\lhat{\infty}.}
\end{proof}

\begin{corollary}\label{cor:deep}
Let $L/K$ be deeply ramified. If $\BB{k}^{\gl}=\lhat{k}$ $($resp.\ $\lhat{\infty}=(\bdr)^{\gl})$, then we have $\widehat{L'}_{k}=\BB{k}^{G_{L'}}$ $($resp. $\widehat{L'}_{\infty}=(\bdr)^{G_{L'}})$ for all algebraic extensions $L'/L$.  
\end{corollary}

\begin{theorem}\label{thm shal}
Let $L/K$ be shallowly ramified.
\begin{enuroman}
\item $\lhat{k}=B_k^{\gl}$ and $\widehat{L}_{\infty}=(\bdr)^{\gl}$. 

\item The valuations $\{w_k\!\!\mid_{L}\}$ of $L$ are equivalent to each other. In particular, we have canonical isomorphisms 
\[
\widehat{L}\cong\lhat{k}\cong {\widehat{L}}_{\infty}
\]
as topological rings.
\end{enuroman}
\end{theorem}
The proof reduces to the following theorem:
\newcommand{\okp}{\mathcal{O}_{K'}}
\newcommand{\olp}{\mathcal{O}_{L'}}
\newcommand{\op}[1]{\mathcal{O}_{#1}}
\newcommand{\mmu}[1]{{#1}^{\mu}}
\newcommand{\mn}[1]{{#1}^n}
\newcommand{\mnlam}[1]{{#1}^n_{\lambda}}
\newcommand{\mlam}[1]{{#1}_{\lambda}}
\newcommand{\mmulam}[1]{{#1}_{\lambda}^{\mu}}
\newcommand{\kah}[2]{{\Omega}^{1}_{#1/#2}}
\def\hsymb#1{\mbox{\strut\rlap{\smash{\Huge$#1$}}\quad}}

\begin{theorem}\label{cohnonram}
If $L/K$ is shallowly ramified, then $J_k^{\gl}=0$ for all $k\in\mathbb{N}$.
\end{theorem}

Let us prove Theorem~\ref{thm shal}, admitting Theorem~\ref{cohnonram}: The equality $\lhat{k}=B_k^{\gl}$ follows from Theorem~\ref{thm:main}. Since $\ker{(\lhat{k}\to\lhat{k-1})}=J_k\cap\lhat{k}\subset J_k^{\gl}=0$, the canonical projection $\lhat{k}\to\lhat{k-1}$ is injective. So, the equivalence of semi-valuations $\{w_k\!\!\mid_{L}\}$ is a direct consequence of Lemma~\ref{valuationw}.

Before the proof of Theorem~\ref{cohnonram}, we need some lemmas. For a while, let $L/K$ be a general algebraic extension.
\begin{nonotation}
Fix $\zeta_{p^n}$ a primitive $p^n$-th root of unity with $\zeta_{p^{n+1}}^p=\zeta_{p^n}$ and put $\varepsilon=(1,\zeta_p,\zeta_{p^2},\dotsc)\in\eplus$. For $x\in\okbar$, put $L(\widetilde{x})=\bigcup_{n}{L(\widetilde{x}^{(n)})}$. For $x\in\ok\setminus\{ 0\}$ and $\widetilde{x}$ as before, let $s_{\widetilde{x}}:\gk\to\zp$ be the map such that $\sigma (\widetilde{x})/\widetilde{x}=\varepsilon^{s_{\widetilde{x}}(\sigma)}$ for $\sigma\in\gk$. $\chi$ denotes the cyclotomic character and $\mu_{p^{\infty}}$ denotes the set of $p$-power roots of unity.
\end{nonotation}

\begin{lemma}\label{fund nonram}
Let $K',L/K$ be linearly disjoint algebraic extensions. Let $\{K^{\mu}\}\ (resp.\ \{L_{\lambda}\})$ be finite subextensions of $K'$ $(resp.\ L)$ over $K$ and put $L_{\lambda}^{\mu}=K^{\mu}\mlam{L}$. Assume $r(\kah{\okp}{\ok})+r(\kah{\ol}{\ok})\le r(\okbar\otimes_{\mathcal{O}_{K'}}\kah{\okp}{\ok}+\okbar\otimes_{\ol}\kah{\ol}{\ok})$. Then there exists a constant $C$ satisfying
\[
0\le v_p(\dif{\mmu{K}}{K})-v_p(\dif{\mmulam{L}}{\mlam{L}})\le C.
\]
Moreover, we have the equality
\[
r(\kah{\okp}{\ok})+r(\kah{\ol}{\ok})=r(\kah{\olp}{\ok}),
\]
where $L'=K'L$.
\end{lemma}
\begin{proof}
Let us prove the first part. The LHS of the asserted inequalities is from the linearly disjointness (see the argument after Definition~\ref{def:deep}). Put $M,N$ $(\mathrm{resp}.\ \mmulam{M},\ \mmulam{N})$ the kernel and cokernel of the canonical morphism of $\olp\otimes\kah{\okp}{\ok}\to\kah{\olp}{\ol}$ $(\mathrm{resp}.\ {{\mathcal{O}}_{\mmulam{L}}\otimes\kah{\mathcal{O}_{\mmu{K}}}{\ok}}\to\kah{{\mathcal{O}}_{\mmulam{L}}}{{\mathcal{O}}_{\mlam{L}}})$.

Put $\Omega=\okbar\otimes_{\mathcal{O}_{K'}}\kah{\okp}{\ok}\oplus\okbar\otimes_{\ol}\kah{\ol}{\ok}$, $\Omega^{'}=\okbar\otimes_{\mathcal{O}_{K'}}\kah{\okp}{\ok}+\okbar\otimes_{\ol}\kah{\ol}{\ok}$ and denotes the fitting ideals of $N_{\lambda}^{\mu}$, $M_{\lambda}^{\mu}$ by $\mathcal{I}(N_{\lambda}^{\mu})$, $\mathcal{I}(M_{\lambda}^{\mu})$. Applying Lemma~\ref{exactness vp} to an exact sequence

\[\xymatrix{
0\ar[r]&\okbar\otimes_{\mathcal{O}_{L'}}M\ar[r]&\Omega\ar[r]&\Omega'\ar[r]&0\\
}\]

\noindent with $\Omega\to\Omega^{'}; (\omega_1,\omega_2)\mapsto\omega_1-\omega_2$, we have $r(\okbar\otimes_{\mathcal{O}_{L'}}M)=r(M)=0$. Since $\{\mmulam{M}\}$ is a direct system of $M$, we have a constant $C$ such that $v_p(\mathcal{I}(\mmulam{M}))<C$. By an exact sequence 
\[\xymatrix{
0 \ar[r]& \mmulam{M} \ar[r] & {\mathcal{O}}_{\mmulam{L}}\otimes\kah{{\mathcal{O}}_{\mmu{K}}}{\ok} \ar[r] & \kah{{\mathcal{O}}_{\mmulam{L}}}{{\mathcal{O}}_{\mlam{L}}} \ar[r] & \mmulam{N} \ar[r] & 0, \\
}\]
\noindent we have the first assertion. Since we have the inequality $v_p(\mathcal{I}(\mmulam{N}))\le v_p(\mathcal{I}(\mmulam{M}))$ by the first assertion and the above exact sequence, the last assertion is from Corollary~\ref{four term lemma}.
\end{proof}

For a while, let $K_0$ is an absolutely unramified local field and assume $K=K_0(\zeta_{p^{n_0}})$ for some $n_0>1$ $(\mathrm{Case}~1)$ or $K=K_0(\widetilde{t}^{(n_0)})$ $(\mathrm{Case}~2)$ for some $n_0\ge 0$ where $t$ is an element of $K_0$ such that $\bar{t}\in k_{K_0}\setminus k_{K_0}^p$. Let us put
\[
K^n=
\begin{cases}
K_0({\zeta}_{p^{n_0+n}}) & (\mathrm{Case}~1) \\
K_0({\widetilde{t}}^{(n_0+n)}) & (\mathrm{Case}~2) \\
\end{cases}
\]
and $K'=\cup K^n$.
\begin{lemma}\label{tate}
$v_p(\dif{K^n}{K})=n.$
\end{lemma}
\begin{proof}
Since, for finite extensions $L_1/L_2/L_3/K$, we have $\dif{L_1}{L_3}=\dif{L_1}{L_2}\dif{L_2}{L_3}$ (\cite[III, $\S 4$, Proposition~8 ]{Se1}), this is a direct consequence of Example~\ref{ex:deep}.
\end{proof}

\begin{lemma}\label{lem:Tatedif}
Let $K,K^n,K'$ be as above and let $L/K$ be an algebraic extension such that the extensions $K',L/K$ satisfy all the assumptions in Lemma~\ref{fund nonram}. Let $\{L_{\lambda}\}$ be the set of finite subextensions of $L/K$ and put $L^n_{\lambda}=K^nL_{\lambda}$, $L^n=K^nL$, $L'=K'L$. Let $\sigma$ be an element of $\gk$ such that $\sigma |_{K^1}\neq\mathrm{id}_{K^1}$ and let $|\cdot |$ be the $p$-adic norm. Then we have the following.
\begin{enuroman}
\item $v_p(\dif{\mnlam{L}}{\mlam{L}})=v_p(\dif{\mn{K}}{K})+b_{\lambda}^{n}$, where $\{b_{\lambda}^n\}_{n}$ is a decreasing sequence and $\{b_{\lambda}^n\}_{n,\lambda}$ is bounded. 

\item $|\tr{L_{\lambda}^{n+1}}{L_{\lambda}^{n}}(x)|\le {|p|}^{1+b_{\lambda}^{n+1} -b_{\lambda}^{n}}|x|$ for $x\in L_{\lambda}^{n+1}$.

\item $| x-p^{-1}\tr{L^{n+1}_{\lambda}}{L^{n}_{\lambda}}(x)|\le |p|^{-1}| x^{{\sigma}^{p^n}}-x|$ for $x\in L^{n+1}_{\lambda}$.

\item If we put $t_{L'/L}=\varinjlim_{n}{{[L^n:L]}^{-1}\tr{L^n}{L}}$, then there exists a constant $C_1$ such that 
\[
|t_{L'/L}(x)-x|\le C_1| x^{\sigma}-x| .
\] 
\end{enuroman}
\end{lemma}
\begin{proof}{The proof is similar to that in {\cite[(3.1)]{Tate}}.}

(i) We have only to prove that $\{b_{\lambda}^n\}_{\lambda}$ is decreasing. Since $L_{\lambda}^n$, $K^{n+1}$ is linearly disjoint over $K^n$, we have $0\le v_p(\dif{K^{n+1}}{K^{n}})-v_p(\dif{L_{\lambda}^{n+1}}{L_{\lambda}^{n}})=b^n_{\lambda}-b^{n+1}_{\lambda}$ by the argument after Definition~\ref{def:deep}.

(ii) is a direct consequence of \cite[III, $\S 3$, Proposition~$7$]{Se1} and Lemma~\ref{tate}, Lemma~\ref{lem:Tatedif}(i).

(iii) Note that, by the assumption on $K$, the set $\{\sigma^{p^ni}|_{L^{n+1}_{\lambda}}\}_{0\le  i<p}$ coincides with the set of conjugate maps of $L^{n+1}_{\lambda}/L_{\lambda}^{n}$. Put $\tau=\sigma^{p^n}$. 
\begin{align*}
px-\mathrm{Tr}_{L^{n+1}_{\lambda}/L^{n}_{\lambda}}(x) & =px-\sum_{0\le i<p}{\tau^{i}x}=\sum_{0\le i< p}{(1-\tau^{i})x}\\
 & =\sum_{1\le i<p}{(1+\tau+\dotsb+\tau^{i-1})(1-\tau)x}.
\end{align*}
Hence $| px-\mathrm{Tr}_{L^{n+1}_{\lambda}/L^n_{\lambda}}(x)|\le| (1-\tau)x|$.

(iv) Put $t_{L^{'}_{\lambda}/L_{\lambda}}=\varinjlim{{[L_{\lambda}^n:L_{\lambda}]}^{-1}\mathrm{Tr}_{L^{n}_{\lambda}/L_{\lambda}}}$. We prove by induction on $n$ an inequality

\[
| x-t_{L^{'}_{\lambda}/L_{\lambda}}(x)|\le c^{n}_{\lambda}| x^{\sigma}-x|\ \text{if}\ x\in L_{\lambda}^n
\]

with $c^{1}_{\lambda}={|p|}^{-1}$, $c^{n+1}_{\lambda}={|p|}^{b_{\lambda}^{n+1}-b_{\lambda}^{n}}c^{n}_{\lambda}$, which implies the assertion. When $n=1$, this is (iii). Assume the above inequality is true for $n$. Then, for $x\in L_{\lambda}^{n+1}$, we have 
\begin{align*}
|\mathrm{Tr}_{L_{\lambda}^{n+1}/L_{\lambda}^{n}}(x)-pt_{L'/L}(x)| & \le c^{n}_{\lambda}| \sigma\mathrm{Tr}_{L_{\lambda}^{n+1}/L_{\lambda}^{n}}(x)-\mathrm{Tr}_{L_{\lambda}^{n+1}/L_{\lambda}^{n}}(x)| \\
 & =c^{n}_{\lambda}| \mathrm{Tr}_{L_{\lambda}^{n+1}/L_{\lambda}^{n}}(x^{\sigma}-x)|\le c^{n}_{\lambda}{|p|}^{1+b_{\lambda}^{n+1}-b_{\lambda}^{n}}| x^{\sigma}-x|
\end{align*}
by (ii). By (iii), we have
\begin{align*}
| x-t_{L'/L}(x)| & \le\sup{(| x-p^{-1}\mathrm{Tr}_{L^{n+1}_{\lambda}/L^n_{\lambda}}(x)|,{|p|}^{b_{\lambda}^{n+1}-b_{\lambda}^n}c^{n}_{\lambda}|x^{\sigma}-x|)}\\
 & \le\sup{(c^{1}_{\lambda},{|p|}^{b_{\lambda}^{n+1}-b_{\lambda}^n}c^{n}_{\lambda})| x^{\sigma}-x|} = {|p|}^{b_{\lambda}^{n+1}-b_{\lambda}^n}c^{n}_{\lambda}| x^{\sigma}-x|.
\end{align*}
Hence the asserted inequality is true for $n+1$.
\end{proof}

\newcommand{\uk}[1]{K^{(#1)}}
\newcommand{\oul}[1]{\mathcal{O}_{L^{(#1)}}}
\newcommand{\ull}[1]{L^{(#1)}}
\newcommand{\okzeta}{\mathcal{O}_{K_0(\zeta_{p^{n_0}})}}
\newcommand{\okuu}[1]{\mathcal{O}_{K^{(#1)}}}
\newcommand{\kzeta}{{K_0}(\zeta_{p^{n_0}})}

\begin{corollary}\label{calc coh simple}
In addition to the assumptions as above, we assume $\mu_{p^{\infty}}\subset L$ in Case 2. Then, $\coh{0}{G_{L'/L}}{\widehat{L'}(n)}=0$ for $n\neq 0$ $($Case 1$)$, $\coh{1}{G_{L'/L}}{\widehat{L'}}=\widehat{L}[s_{\widetilde{t}}]$ $($Case 2$)$, where $(n)$ denotes the Tate twist by $\chi^{n}$  and $[*]$ denotes a cohomology class of $*$.
\end{corollary}
\begin{proof}
First, note that $t_{L'/L}:L'\to L$ extends to a continuous surjective $\widehat{L}$-linear map $t_{L'/L}:\widehat{L'}\to \widehat{L}$. Hence, by applying \cite[Proposition~8]{Tate}, we have only to prove that the cohomology class of $s_{\widetilde{t}}$ does not vanish in Case~2. Since $t_{L'/L}$ kills $1$-coboundaries $\mathrm{B}^{1}_{\text{cont}}(G_{L'/L},\widehat{L'})$ and $t_{L'/L}(s_{\widetilde{t}})=s_{\widetilde{t}}\in \mathrm{Z}_{\mathrm{cont}}^{1}(G_{L'/L},\zp)$, this follows from $s_{\widetilde{t}}\not\equiv 0$.
\end{proof}

\begin{proof}[Proof of Theorem~\ref{cohnonram}]
We only have to prove the theorem under the assumption $\zeta_{p^2}\in K\subset L$. Choose a $p$-basis $t_1,\dotsc,t_d$ of $K_0$ and put $K^{(j)}=K_0(\mu_{p^{\infty}})(\widetilde{t_1},\dotsc ,\widetilde{t_j})$, $L^{(j)}=K^{(j)}L$ and $s_j=s_{\widetilde{t_{j}}}$ for $0\le j\le d$. We claim that
\[
\begin{cases}
K^{(0)}\cap L=K_0(\zeta_{p^{n_0}}) \\
r(\kah{\mathcal{O}_{L^{(0)}}}{\ok})=1, \\
\end{cases}
\begin{cases}
K^{(j+1)}\cap L^{(j)}=K^{(j)}(\widetilde{t_{j+1}}^{(n_{j+1})}) \\
r(\kah{\mathcal{O}_{L^{(j+1)}}}{\ok})=j+2 \\
\end{cases}
\]
for some $n_j\in\mathbb{N}$. Moreover, we claim that $\coh{0}{G_{L^{(0)}/L}}{\widehat{L^{(0)}}(n)}=0$ for $n\neq 0$ and $\coh{1}{G_{L^{(d)}/L^{(0)}}}{\widehat{L^{(d)}}}=\oplus_{1\le j\le d}{\widehat{L^{(0)}}[s_j]}$.

Let us prove this claim. By the hypothesis on $K$, if $K^{(0)}\cap L\neq K_0(\zeta_{p^{n_0}})$ for all $n_0$, then $K^{(0)}\subset L$. Then we have $0=r(\threekahler{1}{\ol}{\ok})\ge r(\threekahler{1}{\mathcal{O}_{K^{(0)}}}{\ok})=1$, which is a contradiction. So we have $K^{(0)}\cap L=K_{0}(\zeta_{p^{n_0}})$ for some $n_0$. Since $K^{(0)}/K_{0}(\zeta_{p^{n_0}})$ is Galois, $K^{(0)}$ and $L$ are linearly disjoint over $K_{0}(\zeta_{p^{n_0}})$ and so we have $r(\threekahler{1}{\okuu{0}}{\okzeta})=r(\threekahler{1}{\okuu{0}}{\ok})=1$, $r(\threekahler{1}{\ol}{\okzeta})=r(\olkahler)=0$, $r(\okbar\otimes\threekahler{1}{\okuu{0}}{\okzeta}+\okbar\otimes\threekahler{1}{\ol}{\okzeta})\ge r(\threekahler{1}{\okuu{0}}{\okzeta})=1$. Therefore $K'=\uk{0}$, $L/\kzeta$ satisfy the assumption of Lemma~\ref{fund nonram}. Applying Lemma~\ref{fund nonram} and Lemma~\ref{calc coh simple}, we have $r(\threekahler{1}{\oul{0}}{\ok})=1$, $\coh{0}{G_{\ull{0}/L}}{\widehat{\ull{0}}(n)}=0$. We prove the rest of the assertion by induction on $j$ ($0\le j<d$). (Instead of proving $\coh{1}{G_{\ull{d}/\ull{0}}}{\widehat{\ull{d}}}=\oplus_{1\le j\le d}{\widehat{\ull{0}}[s_j]}$, we prove $\coh{1}{G_{\ull{j+1}/\ull{0}}}{\widehat{\ull{j+1}}}=\oplus_{1\le i\le j+1}{\widehat{\ull{0}}[s_i]}$.) Since $r(\threekahler{1}{\mathcal{O}_{K^{(j+1)}}}{\ok})=j+2>j+1=r(\threekahler{1}{\mathcal{O}_{L^{(j)}}}{\ok})$ by Example~\ref{ex:deep} and the induction hypothesis, we have $K^{(j+1)}\not\subset L^{(j)}$, i.e., $K^{(j+1)}\cap L^{(j)}/K^{(j)}$ is finite. Put $K^{(j+1)}\cap L^{(j)}=K^{(j)}({\widetilde{t_{j+1}}}^{(n_{j+1})})$, then $L^{(j)}$, $K(\widetilde{t_{j+1}})/K({\widetilde{t_{j+1}}}^{(n_{j+1})})$ satisfy the assumption of Lemma~\ref{fund nonram}. As a consequence of Lemma~\ref{fund nonram} and Corollary~\ref{calc coh simple}, we have $r(\threekahler{1}{\mathcal{O}_{L^{(j+1)}}}{\ok})=j+2$ and $\coh{1}{G_{L^{(j+1)}/L^{(j)}}}{\widehat{L^{(j+1)}}}=\widehat{L^{(j)}}[s_{j+1}]$. In the inflation-restriction sequence
\[\xymatrix{
0\ar[r]&\coh{1}{G_{L^{(j+1)}/L^{(0)}}}{\widehat{L^{(j)}}}\ar[r]&\coh{1}{G_{L^{(j+1)}/L^{(0)}}}{\widehat{L^{(j+1)}}}\ar[r]&{\coh{1}{G_{L^{(j+1)}/L^{(j)}}}{\widehat{L^{(j+1)}}}}^{G_{L^{(j)}/L^{(0)}}},
}\]
the base $[s_{j+1}]$ of the last term lifts to the middle term, i.e., the right arrow is surjective. So, the assertion follows from the induction hypothesis. 

Now we have $J_1=\oplus_{0\le j\le d}{\cp v_j}$, where the action of $\gk$ on $\{v_j\}$ is given by

$$
\left(
\begin{array}{cccc}
\chi&-t_1s_1&\dotsm&-t_ds_d\\
&1&&\\
&&\ddots& \\
&&&1
\end{array}
\right)(\text{empty entries are }0)
$$

\noindent and $J_k=\mathrm{Sym}^k_{\cp} J_1=\oplus_{n\in\mathbb{N}_{k}^{d+1}}{\cp v^n}$ where $\mathbb{N}_{k}^{d+1}=\{n\in\mathbb{N}^{d+1}| | n |=k\}$ and $v^n=v_0^{n_0}\dotsm v_d^{n_d}$ (\cite[Just before Lemma~2.1.12]{Bri1}). Obviously, $\coh{0}{G_{L^{(d)}}}{J_k}=\oplus_{n\in\mathbb{N}_{k}^{d+1}}{\widehat{L^{(d)}} v^n}$. Let $\succ$ be the order on $\mathbb{N}_{k}^{d+1}$ defined by

\[
n^1\succ n^2\Leftrightarrow n^1_0\ge n^2_0,\ n^1_1\le n_1^2,\dotsc,\ n_d^1\le n^2_d. 
\]

For $n'\in\mathbb{N}_{k-1}^{d+1}$, put $N_{n'}=\bigcup_{0\le j\le d}{\{n\in {\mathbb{N}}_{k}^{d+1} | n\succ n'+\mathbf{e}_j\}}$, $N_{n'}^{\circ}=\bigcup_{0\le j\le d}{\{n\in\mathbb{N}_{k}^{d+1}|n\succneqq n'+\mathbf{e}_j\}}$, where $\mathbf{e}_j=(0,\dotsc,\stackrel{j}{\check{1}},\dotsc,0)$ for $0\le j\le d$. Since $\oplus_{n\in N_{n'}}{\widehat{L^{(d)}}v^n}$ and $\oplus_{n\in N_{n'}^{\circ}}{\widehat{L^{(d)}}v^n}$ is $G_{L^{(d)}/L}$-stable, we have an $\widehat{L^{(d)}}$-representation of $G_{L^{(d)}/L}$

\[
V_{n'}=\oplus_{n\in N_{n'}}{\widehat{L^{(d)}}v^n}/\oplus_{n\in N_{n'}^{\circ}}{\widehat{L^{(d)}}v^n}\cong \oplus_{0\le j\le d}{\widehat{L^{(d)}}{\overline{v}}^{n'+\mathbf{e}_j}}
\]

\noindent and the Galois action on $\{\overline{v}^{n'+\mathbf{e}_j}\}$ is given by 

$$
\left(
\begin{array}{cccc}
\chi^{n'_0+1}&-t_1s_1\chi^{n'_0}&\dotsm&-t_ds_d\chi^{n'_0}\\
&\chi^{n'_0}&&\\
&&\ddots& \\
&&&\chi^{n'_0}
\end{array}
\right)(\text{empty entries are }0).
$$

\noindent Then $\coh{0}{G_{L^{(d)}/L}}{V_{n'}}=0$: Let $x=\sum_{0\le j\le d}{x_j\overline{v}^{n'+\mathbf{e}_j}}\in\coh{0}{G_{L^{(d)}/L}}{V_{n'}}$. Restricting to $G_{L^{(d)}/L^{(0)}}$, we have $x_j\in\widehat{L^{(0)}}$ for $j>0$ and $(\sigma-1)x_0=\sum_{j>0}{x_jt_js_j(\sigma)}$ for $\sigma\in G_{L^{(d)}/L^{(0)}}$. Since $\coh{1}{G_{L^{(d)}/L^{(0)}}}{\widehat{L^{(d)}}}$ is an $\widehat{L^{(0)}}$-vector space with a basis $[s_1],\dotsc,[s_d]$, we have $x_j=0$ for $j>0$, hence $x_0\in\widehat{L^{(0)}}$. Also, we have $x_0(\mathrm{log} [\varepsilon])^{n'_0+1}\in\coh{0}{G_{L^{(0)}/L}}{\widehat{L^{(0)}}(n'_0+1)}$=0, i.e., $x_0=0$.

Let us finish the proof. Let $x=\sum_{n\in\mathbb{N}_{k}^{d+1}}{x_nv^n}\in J_k^{\gl}$ with $x_n\in\widehat{{L}^{(d)}}$. To prove $x=0$, we have only to prove that, for all $n'\in\mathbb{N}_{k-1}^{d+1}$, we have $x_{n'+\mathbf{e}_j}=0$ for $0\le j\le d$. If not, choose a minimal (with respect to the above order $\succ$) $n'\in\mathbb{N}_{k-1}^{d+1}$ with $x_{n'+\mathbf{e}_j}\neq 0$ for some $j$. Then the image of $\sum_{n\in N_{n'}}{x_nv^n}$ in $V_{n'}$ is contained in $\coh{0}{G_{L^{(d)}/L}}{V_{n'}}$ by the minimality of $n'$. Hence we have $x_{n'+\mathbf{e}_j}=0$ for $0\le j\le d$, which is a contradiction.
\end{proof}

Finally, we describe some concrete examples.

\begin{example}
Fix a uniformizer $\pi_K$ of $K$, a $p$-basis $t_1,\dotsc,t_d$ of $K$ and put $S_m(X)=X^{p^m}+\pi_{K}X$. Let $L/K$ be an algebraic extension generated by all roots of $S_m(X)=\widetilde{\pi_K}^{(n)}$, $S_m(X)=\widetilde{t_j}^{(n)}$, $1\le j\le d$, for all $n,m$. We prove that $L/K$ is deeply ramified and satisfies $\widehat{L}_k=B_k^{\gl}$ for all $k$ and $\widehat{L}_{\infty}=(\bdr)^{\gl}$. For $n\in\mathbb{N}_{>0}$, choose $x_n^0,x_n^1,\dotsc,x_n^d$ such that $S_1(x_n^0)=\widetilde{\pi_K}^{(n)}$, $S_1(x_n^1)=\widetilde{t_j}^{(n)},\dotsc,S_1(x_n^d)={\widetilde{t_d}}^{(n)}$ and put $L_n=K(\pi_n^0,\dotsc,\pi_n^d)$. Then, $\mathcal{O}_{L_n}=\ok [\pi_n^0,\dotsc,\pi_n^d]=\ok[\pi^0_n]\otimes\dotsb\otimes\ok[\pi_n^d]$ and the minimal polynomial of $\pi_n^0$ (resp. $\pi_n^j$) over $\ok$ is $S_1(X)^{p^n}-\pi_K$ (resp. $S_1(X)^{p^n}-t_j$). It is easy to see that the $p$-adic valuation of the unique invariant factor of $\threekahler{1}{\ok[\pi_n^j]}{\ok}$ is at least $n$. Hence $L/K$ is deeply ramified. By the proof of Lemma~\ref{lemdense}, we see that $\widehat{L}_k$ contains $[\widetilde{\pi_K}],[\widetilde{t_1}],\dotsc,[\widetilde{t_d}]$, therefore contains $\pi_K-[\widetilde{\pi_K}],t_1-[\widetilde{t_1}],\dotsc,t_d-[\widetilde{t_d}]$, which generates $I/I^{k+1}$.
\end{example}

The following example is a generalization of \cite[Corollary~8.2]{IZ}.
\begin{example}\label{ex:dR}
Let $m\in\mathbb{N}_{>0}$ and assume $m>1$ or $e_K>1$. Let $q=p^m$ and put $f=X^q+\pi_K X$, $f^n=f\circ\dotsb\circ f$ ($n$-times). Put $\pi^0_0=0$ and let $\pi_0^1=t_1,\dotsc,\pi_0^d=t_d$ be a $p$-basis of $K$. Let $K_n/K$ be an algebraic extension generated by all roots of $f^n(X)=\pi^0_j$ for $0\le j\le d$ and $K_{\infty}=\cup{K_n}$. In the following, we prove that $K_{\infty}/K$ is deeply ramified and de Rham at level $1$.

For $n\in\mathbb{N}_{>0}$, fix $\pi_n^j$ for $0\le j\le d$ with $\pi_1^0\neq 0$ such that $f(\pi_n^j)=\pi_{n-1}^j$ and put $L_n=K(\pi_n^j|\ 0\le j\le d,\ n\in\mathbb{N})$, $L=\cup{L_n}$. We have only to prove the assertion for $L/K$ by Corollary~\ref{cor:deep}. 

By a simple calculation, we have 
\begin{align*}
\mathcal{O}_{L_n}=\ok[\pi_n^0,\dotsc,\pi_n^d] & =\ok[\pi_n^0]\otimes_{\ok}\dotsb\otimes_{\ok}\ok[\pi_n^d],\\
\threekahler{1}{\ok[\pi_n^j]}{\ok} & =(\ok[\pi_n^j]/\pi_K^{n-\delta_j})d\pi_n^j,\end{align*}
where $\delta_0=1/(q-1)$, $\delta_j=0$ for $0\le j\le d$. In particular, $L/K$ is deeply ramified. On the other hand, by taking derivation of the equation $f(\pi_{n+1}^j)=\pi_n^j$, we obtain $q(\pi_{n+1}^j)^{q-1}d{\pi}_{n+1}^j+\pi_K d{\pi}_{n+1}^j=d\pi_n^j$. Hence, for $\lambda\in\mathcal{O}_{L}$ with $v_K(\lambda)\ge n+1-\delta_j-e_K m\ge v_K(\mathrm{Ann}(q(\pi_{n+1}^j)^{q-1}d\pi_{n+1}^j))$, we have
\begin{equation}\label{equ:val}
\lambda d\pi_n^j=\pi_K\lambda d\pi_{n+1}^j.
\end{equation}
Put $L^j=K(\pi_n^j|\ n\in\mathbb{N})$ and we claim that, for $\omega\in\threekahler{1}{\mathcal{O}_{L^j}}{\ok}$ and $k\in\mathbb{N}$, there exists $x\in\mathcal{O}_{L^j}$ such that $\omega=\pi_K^kdx$. Let us prove this claim by induction on $[v_K(\mathrm{Ann}(\omega))]$: Assume $\pi_K\omega=0$. Write $\omega=\lambda d\pi_n^j$ for some $n\in\mathbb{N}$, $\lambda\in\mathcal{O}_{L^j}$ and considering the annihilator of both sides, we have $v_K(\lambda)\ge v_K(\mathrm{Ann}(d\pi_{n}^j))-v_K(\mathrm{Ann}(\omega))\ge n-\delta_j-1\ge n+1-\delta_j-e_Km$. Applying (\ref{equ:val}), we have $\lambda d\pi_n^j=\pi_K\lambda d\pi_{n+1}^j$ and iterating this procedure and taking $m$ sufficiently large with $\pi_K^{m-k}\lambda\in\,\mathcal{O}_{L}^{(1)}$, we have $\omega=\pi_K^m\lambda d\pi_{n+m}^j=\pi_K^k d(\pi_K^{m-k}\lambda d\pi_{n+m}^j)$. For general $\omega$, applying the induction hypothesis to $\pi_K\omega$, we have $\pi_K\omega =\pi_K^{k+1}dx$ and again applying the induction hypothesis to $\omega-\pi_K^k dx$, we have the conclusion. In particular, $d:\mathcal{O}_{L^j}\to\threekahler{1}{\mathcal{O}_{L^j}}{\ok}$ is surjective and this implies $\drcoh{1}=0$ by Remark~\ref{rem:dr}(ii) since $\threekahler{1}{\ol}{\ok}=\oplus_{0\le j\le d}{\ol\otimes\threekahler{1}{\mathcal{O}_{L^j}}{\ok}}$ and $\threekahler{1}{\mathcal{O}_{L^j}}{\ok}$ is $p$-divisible.
\end{example}

\begin{example}\label{ex:non dR}
Assume $K=K_0$. Let $t_1,\dotsc,t_d\in K$ be a $p$-basis and put $K_n=K({\zeta}_{p^n},{\widetilde{t_1}}^{(n)},\dotsc,{\widetilde{t_d}}^{(n)})$, $L=\cup K_n$. By Example~\ref{ex:deep}, $L/K$ is deeply ramified. We will show $L/K$ is not de Rham at level $1$. To prove this, we only have to prove $\mathrm{Im}(d:\ol\to \threekahler{1}{\ol}{\ok})_{\mathrm{div}}=0$. In fact, we prove a finer statement:   
\end{example}
\newcommand{\okn}[1]{\mathcal{O}_{K_{#1}}}
\newcommand{\dlog}{\mathrm{dlog}}
\begin{proposition}
Let $\omega\in\threekahler{1}{\okn{n}}{\ok}$, $n\ge 1$. If there exists $m\ge n+1$ and $k\in\mathbb{N}$ such that $\omega\in p^k\mathrm{Im}(d:\okn{m}\to \threekahler{1}{\okn{m}}{\ok})$, then $\omega\in p^k\mathrm{Im}(d:\okn{n}\to\threekahler{1}{\okn{n}}{\ok})$.
\end{proposition}
\begin{proof}
We will prove $\omega\in p^k\mathrm{Im}(d:\okn{m-1}\to\threekahler{1}{\okn{m-1}}{\ok})$. Put $\varepsilon_0=1/(p-1)$, $\varepsilon_j=0$ for $0 < j\le d$ and put ${\omega}_{m}^{j}\in\threekahler{1}{\okn{m}}{\ok}$ as $\omega_m^0=\dlog {\zeta}_{p^m}$, ${\omega}_m^j=\dlog {\widetilde{t_j}}^{(m)}$ for $0< j\le d$ and put $f_0=p^{m-1}(p-1)$, $f_j=p^m$ for $0< j\le d$. By Example~\ref{ex:deep}, we identify

\[
\mathrm{Im}(\threekahler{1}{\okn{n}}{\ok}\to\threekahler{1}{\mathcal{O}_{K_m}}{\ok})=\oplus_{0\le j\le d}{(p^{m-n}\okn{n}+p^{m-\varepsilon_j}\okn{m})/p^{m-\varepsilon_j}\okn{m}\cdot\omega_m^j}.
\]
 
Let $x\in\okn{m}$ such that $p^kdx=\omega$. Let $\omega=\sum_{j}{\lambda_j\omega^j_n}=\sum_{j}p^{m-n}\lambda_j\omega_m^j$ with $\lambda_j\in\okn{n}$. Writing $x=\sum_{0\le e<f}{a_e\pi^e}$ with $a_e\in\ok$ and $\pi^e=\zeta_{p^m}^{e_0}\widetilde{t_1}^{(m)e_1}\dotsb\widetilde{t_d}^{(m)e_d}$ and considering $dx$, we have 
\begin{equation}\label{equ:def}
p^k\sum_{e}{e_ja_e\pi^e}\in p^{m-n}\lambda_j+p^{m-\varepsilon_j}\okn{m} 
\end{equation}
for $0\le j\le d$.

$\mathcal{O}_{K_m}$ has a basis $T=\{\pi^e |\ 0\le e<f\}$ as a free $\ok$-module and $\mathcal{O}_{K_{m-1}}$ has a basis $T_1=\{\pi^e |\ 0\le e <f,\ p|e_j\ \text{for all}\ j\}$ over $\ok$. Let $V$ be the free $\ok$-module spanned by $T\setminus T_1$. Then the direct sum $\mathcal{O}_{K_m}=\mathcal{O}_{K_{m-1}}\oplus V$ is stable under multiplication by $\zeta_p$, so is multiplication by $p^{m-\varepsilon_j}$. Writing both sides of (\ref{equ:def}) as a sum $\mathcal{O}_{K_{m-1}}\oplus V$ and looking at the first factor, we obtain

\[
p^k\sum_{p|e_0,\dotsc,e_d}{e_ja_e\pi^e}\in p^{m-n}\lambda_j+p^{m-\varepsilon_j}\okn{m-1}.
\]

Hence $y=\sum_{p|e_0,\dotsc,e_d}{a_e\pi^e}\in\okn{m-1}$ satisfies $p^kdy=\omega$.
\end{proof}
Now we know $\widehat{L}_1\neq B_1^{\gl}$. As for $\widehat{L}_1$ itself, the canonical projection $\mathrm{pr}:\widehat{L}_1\to\widehat{L}$ injective: For a general algebraic extension $L/K$, we have an exact sequence
\[\xymatrix{
0\ar[r]&p^n\mathcal{O}_{L}^{(k-1)}/p^n\mathcal{O}_{L}^{(k)}\ar[r]&L/p^n\mathcal{O}_{L}^{(k)}\ar[r]&L/p^n\mathcal{O}_{L}^{(k-1)}\ar[r]&0\\
}\]
with $\mathrm{Im}(d^{(k)}:\mathcal{O}_{L}^{(k-1)}\to\threekahler{(k)}{\ol}{\ok})=\mathcal{O}_{L}^{(k-1)}/\mathcal{O}_{L}^{(k)}\cong p^n\mathcal{O}_{L}^{(k-1)}/p^n\mathcal{O}_{L}^{(k)}$. Passing to the limit, we have an exact sequence
\[\xymatrix{
0\ar[r]&\varprojlim{d^{(k)}(\mathcal{O}_{L}^{(k-1)})}\ar[r]&\widehat{L}_k\ar[r]&\widehat{L}_{k-1}\ar[r]&\varprojlim^{1}{d^{(k)}(\mathcal{O}_{L}^{(k-1)})}\ar[r]&0\\
}\]
where the inverse limit is taken for multiplication by $p$. Note that $\varprojlim{d^{(k)}(\mathcal{O}_{L}^{(k-1)})}=0$ if $\mathrm{Im}(d^{(k)}:\mathcal{O}_{L}^{(k-1)}\to\threekahler{(k)}{\ol}{\ok})_{\mathrm{div}}=0$.

Applying this to our $L$ and $k=1$, we see that $\mathrm{pr}:\widehat{L}_1\to\widehat{L}$ is injective since $\mathrm{Im}(d:\ol\to\threekahler{1}{\ol}{\ok})_{\mathrm{div}}=0$. Moreover $\mathrm{pr}$ is not surjective: If not, $\mathrm{pr}$ would be an isomorphism on $p$-adic Banach spaces by the open mapping theorem of $p$-adic Banach spaces. However, $\{p^n\zeta_{p^n}\}$ is a $p$-adic Cauchy sequence, which is not $B_1$-Cauchy: If there exists sufficiently large $n<m\in\mathbb{N}$ such that $p^m\zeta_{p^m}-p^n\zeta_{p^n}\in p^2\mathcal{O}_{L}^{(1)}$ then, by a simple calculation, we have $p^{m-2}(\zeta_{p^m}-\zeta_{p^n})\dlog \zeta_{p^m}=0$. This implies $v_p(p^{m-2}(\zeta_{p^m}-\zeta_{p^n}))=m-2+1/p^{m-1}(p-1)\ge v_{p}(\mathrm{Ann}(\dlog \zeta_{p^m}))=m-1/(p-1)$ by Example~\ref{ex:deep}, which is a contradiction.

\begin{remark}
Let $[K:\qp]<\infty$. As a consequence of the Lubin-Tate theory, Example~\ref{ex:dR} and \ref{ex:non dR}, $K^{\mathrm{ab}}/K$ is deeply ramified and 
\[
K^{\mathrm{ab}}/K\text{ is de Rham at level }1\Leftrightarrow K\neq\qp. 
\]
\end{remark}

\end{document}